\documentclass[11pt]{article}
\usepackage{amssymb}
\usepackage{amsmath}
\usepackage{amsthm} 
\usepackage{times}
\usepackage[mathscr]{eucal}
\usepackage{mathrsfs}

\newtheorem{lemma}{Lemma}[section]

\newtheorem{theorem}{Theorem}[section]

\allowdisplaybreaks[4] 

\renewenvironment{proof}{
\noindent{\bf Proof.} \rm}{\penalty-20\null\hfill $\square$}

\numberwithin{equation}{section}

\newcommand{\bbD}{{\mathbb D}}

\newcommand{\N}{{\mathbb N}}

\newcommand{\R}{{\mathbb R}}
\newcommand{\bbT}{{\mathbb T}}

\newcommand{\bfe}{\mathbf{e}}
\newcommand{\bff}{\mathbf{f}}
\newcommand{\bfg}{\mathbf{g}}
\newcommand{\bfh}{\mathbf{h}}

\newcommand{\bfn}{\mathbf{n}}

\newcommand{\bfu}{\mathbf{u}}
\newcommand{\bfv}{\mathbf{v}}
\newcommand{\bfw}{\mathbf{w}}
\newcommand{\bfx}{\mathbf{x}}
\newcommand{\bfy}{\mathbf{y}}
\newcommand{\bfz}{\mathbf{z}}

\newcommand{\bfzeta}{\mbox{\boldmath $\zeta$}}

\newcommand{\bfphi}{\mbox{\boldmath $\phi$}}

\newcommand{\bfvarphi}{\mbox{\boldmath $\varphi$}}
\newcommand{\bfpsi}{\mbox{\boldmath $\psi$}}

\newcommand{\bfC}{\mathbf{C}}

\newcommand{\bfE}{\mathbf{E}}
\newcommand{\bfF}{\mathbf{F}}
\newcommand{\bfG}{\mathbf{G}}

\newcommand{\bfL}{\mathbf{L}}

\newcommand{\bfU}{\mathbf{U}}

\newcommand{\bfW}{\mathbf{W}}

\newcommand{\cA}{{\cal A}}
\newcommand{\cB}{{\cal B}}

\newcommand{\cF}{{\cal F}}

\newcommand{\cP}{{\cal P}}

\newcommand{\gB}{\mathfrak{B}}

\newcommand{\rmd}{\mathrm{d}}
\newcommand{\rme}{\mathrm{e}}

\renewcommand{\div}{\mathrm{div}\,}

\newcommand{\bfzero}{\mathbf{0}}

\newcommand{\supp}{\mathrm{supp}\,}

\newcommand{\curl}{\mathbf{curl}\, }
\newcommand{\br}{\hbox to 0.7pt{}}

\newcommand{\bbTd}{\bbT_{\hspace{-0.5pt} \rm d}}
\newcommand{\blangle}{\bigl\langle}
\newcommand{\brangle}{\bigr\rangle}
\newcommand{\bllangle}{\blangle\hspace{-2.5pt}\blangle}
\newcommand{\brrangle}{\brangle\hspace{-2.5pt}\brangle}
\newcommand{\Blangle}{\Bigl\langle}
\newcommand{\Brangle}{\Bigr\rangle}

\newcommand{\llangle}{\langle\hspace{-1.9pt}\langle}
\newcommand{\rrangle}{\rangle\hspace{-1.9pt}\rangle}

\newcommand{\Cns}{\bfC^{\infty}_{0,\sigma}}
\newcommand{\Ls}{\bfL_{\tau,\sigma}}
\newcommand{\Ps}[1]{P_{#1}}

\newcommand{\cPs}[1]{\cP_{#1}}

\newcommand{\Wt}{\bfW_{\tau}}
\newcommand{\Wts}{\bfW_{\tau,\br\sigma}}
\newcommand{\Wtc}{\bfW_{\tau,\br c}}
\newcommand{\Wtsc}{\bfW_{\tau,\br\sigma,\br c}}

\newcommand{\Aq}{A_q}

\newcounter{constants}
\newcommand{\cn}[2]{\addtocounter{constants}{1}
\newcounter{c#1#2} \setcounter{c#1#2}{\value{constants}}}
\newcommand{\cc}[2]{c_{\arabic{c#1#2}}}

\begin{document}

\title{\LARGE \bf A Pressure Associated with a Weak Solution to
the Navier--Stokes Equations with Navier's Boundary Conditions}

\author{Ji\v{r}\'{\i} Neustupa, \ \v{S}\'arka Ne\v{c}asov\'a, \
Petr Ku\v{c}era \footnote{Authors' address: Czech Academy of
Sciences, Institute of Mathematics, \v{Z}itn\'a 25, 115 67 Praha
1, Czech Re\-pub\-lic, e--mails: neustupa@math.cas.cz,
matus@math.cas.cz, petr.kucera@cvut.cz}}

\date{}

\maketitle

\begin{abstract}
We show that if $\bfu$ is a weak solution to the Navier--Stokes
initial--boundary value problem with Navier's slip boundary
conditions in $Q_T:=\Omega\times(0,T)$, where $\Omega$ is a domain
in $\R^3$, then an associated pressure $p$ exists as a
distribution with a certain structure. Furthermore, we also show
that if $\Omega$ is a ``smooth'' domain in $\R^3$ then the
pressure is represented by a function in $Q_T$ with a certain rate
of integrability. Finally, we study the regularity of the pressure
in sub-domains of $Q_T$, where $\bfu$ satisfies Serrin's
integrability conditions.
\end{abstract}

\vspace{1mm} \noindent
{\it AMS math.~classification (2010):} \


\section{Introduction} \label{S1}

{\bf 1.1. The Navier--Stokes initial--boundary value problem with
Navier's boundary conditions.} \ Let $T>0$ and $\Omega$ be a
locally Lipschitz domain in $\R^3$, satisfying the condition

\begin{list}{}
{\setlength{\topsep 2pt}
\setlength{\itemsep 1pt}
\setlength{\leftmargin 18pt}
\setlength{\rightmargin 0pt}
\setlength{\labelwidth 10pt}}

\item[(i)]
{\it there exists a sequence of bounded Lipschitz domains
$\Omega_1\subseteq\Omega_2\subseteq\dots$ such that
$\Omega=\bigcup_{n=1}^{\infty}\Omega_n$ and
$(\partial\Omega_n\cap\Omega)\subset\{\bfx\in\R^3;\ |\bfx|\geq
n\}$ for all $n\in\N$.}

\end{list}

\noindent
Note that condition (i) is automatically satisfied e.g.~if
$\Omega=\R^3$ or $\Omega$ is a half-space in $\R^3$ or $\Omega$ is
a bounded or exterior Lipschitz domain in $\R^3$. Put
$Q_T:=\Omega\times(0,T)$ and $\Gamma_T:=
\partial\Omega\times(0,T)$. We deal with the Navier--Stokes system
\begin{align}
\partial_t\bfu+\bfu\cdot\nabla\bfu+\nabla p\ &=\ \nu\Delta\bfu+\bff &&
\mbox{in}\ Q_T, \label{1.1} \\
\div\bfu\ &=\ 0 && \mbox{in}\ Q_T \label{1.2}
\end{align}
with the slip boundary conditions
\begin{equation}
\mbox{a)} \quad \bfu\cdot\bfn=0, \qquad \mbox{b)} \quad
[\bbTd(\bfu)\cdot\bfn]_{\tau}+\gamma\bfu=\bfzero \qquad \mbox{on}\
\Gamma_T \label{1.3}
\end{equation}
and the initial condition
\begin{equation}
\bfu\, \bigl|_{t=0} \bigr.\ =\ \bfu_0. \label{1.4}
\end{equation}
Equations (\ref{1.1}), (\ref{1.2}) describe the motion of a
viscous incompressible fluid in domain $\Omega$ in the time
interval $(0,T)$. The unknowns are $\bfu$ (the velocity) and $p$
(the pressure). Factor $\nu$ in equation (\ref{1.1}) denotes the
kinematic coefficient of viscosity (it is supposed to be a
positive constant) and $\bff$ denotes an external body force. The
outer normal vector field on $\Omega$ is denoted by $\bfn$,
$\bbTd(\bfu)$ denotes the dynamic stress tensor,
$-\bbTd(\bfu)\cdot\bfn$ is the force with which the fluid acts on
the boundary of $\Omega$ (we put the minus sign in front of
$\bbTd(\bfu)\cdot\bfn$ because $\bfn$ is the outer normal vector
and we express the force acting on $\partial\Omega$ from the
interior of $\Omega$), subscript $\tau$ denotes the tangential
component and $\gamma$ (which is supposed to be a nonnegative
constant) is the coefficient of friction between the fluid and the
boundary of $\Omega$. The density of the fluid is supposed to be
constant and equal to one. In an incompressible Newtonian fluid,
the dynamic stress tensor satisfies
$\bbTd(\bfu)=2\nu\br\bbD(\bfu)$, where the rate of deformation
tensor $\bbD(\bfu)$ equals $(\nabla\bfu)_s$ (the symmetric part of
$\nabla\bfu$).

Equations (\ref{1.1}), (\ref{1.2}) are mostly studied together
with the no--slip boundary condition
\begin{equation}
\bfu\ =\ \bfzero \label{1.5}
\end{equation}
on $\Gamma_T$. However, an increasing attention in recent years
has also been given to boundary conditions (\ref{1.3}), which have
a good physical sense. While condition (\ref{1.3}a) expresses the
impermeability of $\partial\Omega$, condition (\ref{1.4}b)
expresses the requirement that the tangential component of the
force with which the fluid acts on the boundary be proportional to
the tangential velocity. Conditions (\ref{1.3}) are mostly called
Navier's boundary conditions, because they were proposed by
H.~Navier in the first half of the 19th century.

\vspace{4pt} \noindent
{\bf 1.2. Briefly on the qualitative theory of the problem
(\ref{1.1})--(\ref{1.4}).} \ As to the qualitative theory for the
problem (\ref{1.1})--(\ref{1.4}), it is necessary to note that it
is not at the moment so elaborated as in the case of the no-slip
boundary condition (\ref{1.5}). Nevertheless, the readers can find
the definition of a weak solution to the problem
(\ref{1.1})--(\ref{1.4}) and the proof of the global in time
existence of a weak solution e.g.~in the papers \cite{ChQi} (with
$\bff=\bfzero$), \cite{NePe2} (in a time-varying domain $\Omega$)
and \cite{Saal} (in a half-space). We repeat the definition in
section \ref{S3}. Theorems on the local in time existence of a
strong solution are proven e.g.~in \cite{ChQi} (for
$\bff=\bfzero$) and \cite{KuNe} (in a smooth bounded domain
$\Omega$). Steady problems are studied in \cite{AmRe1} and
\cite{AmRe2}.

\vspace{4pt} \noindent
{\bf 1.3. On the contents and results of this paper.} \ We shall
see in section \ref{S3} that the definition of a weak solution to
the problem (\ref{1.1})--(\ref{1.4}) does not explicitly contain
the pressure. (This situation is well known from the theory of the
Navier--Stokes equations with the no--slip boundary condition
(\ref{1.5}).) This is also why we usually understand, under a
``weak solution'', only the velocity $\bfu$ and not the pair
$(\bfu,p)$. There arises a question whether one can naturally
assign some pressure $p$ to a weak solution $\bfu$. It is known
from the theory of the Navier--Stokes equations with the no--slip
boundary condition (\ref{1.5}) that the pressure, associated with
a weak solution, generally exists only as a distribution in $Q_T$.
(See \cite{Li}, \cite{Te}, \cite{Si}, \cite{Ga2}, \cite{So},
\cite{Wo} and \cite{Ne2}.) The distribution is regular (i.e.~it
can be identified with a function with some rate of integrability
in $Q_T$) if domain $\Omega$ is ``smooth'', see \cite{SoWa},
\cite{GiSo} and \cite{Ne2}. In section \ref{S4} of this paper, we
show that one can naturally assign a pressure, as a distribution,
to a weak solution to the Navier--Stokes equations with Navier's
boundary conditions (\ref{1.3}), too. Moreover, we show in section
\ref{S4} that the associated pressure is not just a distribution,
satisfying together with the weak solution $\bfu$ equations
(\ref{1.1}), (\ref{1.2}) in the sense of distributions in $Q_T$
(where the distributions are applied to test functions from
$\bfC^{\infty}_0(Q_T)$), but that it is a distribution with a
certain structure, which can be applied to functions from
$\bfC^{\infty}(\overline{Q_T})$ with a compact support in
$\overline{\Omega}\times(0,T)$ and with the normal component equal
to zero on $\Gamma_T$. In section \ref{S5}, we show that if domain
$\Omega$ is smooth and bounded then the associated pressure is a
function with a certain rate of integrability in $Q_T$. Finally,
in section \ref{S6}, we study the regularity of the associated
pressure in a sub-domain $\Omega'\times(t_1,t_2)$ of $Q_T$, where
$\bfu$ satisfies Serrin's integrability conditions. We shall see
that the regularity depends on boundary conditions, satisfied by
the velocity on $\Gamma_T$.


\section{Notation and auxiliary results} \label{S2}

{\bf 2.1. Notation.} \ We use this notation of functions, function
spaces, dual spaces, etc.:

\begin{list}{$\circ$}
{\setlength{\topsep 2pt}
\setlength{\itemsep 1pt}
\setlength{\leftmargin 14pt}
\setlength{\rightmargin 0pt}
\setlength{\labelwidth 6pt}}

\item
$\Omega_0\subset\subset\Omega$ means that $\Omega_0$ is a bounded
domain in $\R^3$ such that $\overline{\Omega_0}\subset\Omega$.

\item
Vector functions and spaces of vector functions are denoted by
boldface letters.

\item
$\Cns(\Omega)$ denotes the linear space of infinitely
differentiable divergence-free vector functions in $\Omega$, with
a compact support in $\Omega$.

\item
Let $1<q<\infty$. We denote by $\Ls^q(\Omega)$ the closure of
$\Cns(\Omega)$ in $\bfL^q(\Omega)$. The subscript $\tau$ means
that functions from $\Ls^q(\Omega)$ have the normal component on
$\partial\Omega$ equal to zero in a certain weak sense of traces
and they are therefore tangential on $\partial\Omega$. The
subscript $\sigma$ expresses the fact that functions from
$\Ls^q(\Omega)$ are divergence--free in $\Omega$ in the sense of
distributions. (See e.g.~\cite{Ga1} for more information.)

\item
Put $\bfG_{q}(\Omega):=\{\nabla\psi\in\bfL^{q} (\Omega);\ \psi\in
W^{1,q}_{\rm loc}(\Omega)\}$. $\bfG_{q}(\Omega)$ is a closed
subspace of $\bfL^{q}(\Omega)$, see \cite[Exercise III.1.2] {Ga1}.

\item
$\Wt^{1,q}(\Omega):=\{\bfv\in\bfW^{1,q}(\Omega);\,
\bfv\cdot\bfn=0\ \mbox{a.e.~on}\ \partial\Omega\}$, \\ [3pt]
$\Wtc^{1,q}(\Omega):=\bigl\{ \bfvarphi\in\Wt^{1,q}(\Omega)$,
$\supp\bfvarphi$ is a compact set in $\R^3\bigr\}$, \\ [3pt]
$\Wts^{1,q}(\Omega):=\bfW^{1,q}(\Omega)\cap\Ls^q(\Omega)\equiv
\Wt^{1,q}(\Omega)\cap\Ls^q(\Omega)$, \\ [3pt]
$\Wtsc^{1,q}(\Omega):=\Wts^{1,q}(\Omega)\cap\Wtc^{1,q}(\Omega)$.

\item
The norms in $L^q(\Omega)$ and in $\bfL^q(\Omega)$ are denoted by
$\|\, .\, \|_q$. The norms in $W^{k,q}(\Omega)$ and in
$\bfW^{k,q}(\Omega)$ (for $k\in\N$) are denoted by $\|\, .\,
\|_{k,q}$. If the considered domain differs from $\Omega$ then we
use e.g.~the notation $\|\, .\, \|_{q;\, \Omega'}$ or $\|\, .\,
\|_{k,q;\, \Omega'}$, etc. The scalar products in $L^2(\Omega)$
and in $\bfL^2(\Omega)$ are denoted by $(\, .\, ,\, .\, )_2$ and
the scalar products in $W^{1,2}(\Omega)$ and in
$\bfW^{1,2}(\Omega)$ are denoted by $(\, .\, ,\, .\, )_{1,2}$.

\item
The conjugate exponent is denoted by prime, so that
e.g.~$q'=q/(q-1)$. $\Wt^{-1,q'}(\Omega)$ denotes the dual space to
$\Wt^{1,q}(\Omega)$ and $\Wts^{-1,q'}(\Omega)$ denotes the dual
space to $\Wts^{1,q}(\Omega)$. The norm in $\Wt^{-1,q'}(\Omega)$,
respectively $\Wts^{-1,q'}(\Omega)$, is denoted by $\|\, .\,
\|_{-1,q'}$, respectively by $\|\, .\, \|_{-1,q';\, \sigma}$.

\item
The duality between elements of $\Wt^{-1,q'}(\Omega)$ and
$\Wt^{1,q}(\Omega)$ is denoted by $\langle\, .\, ,\, .\,
\rangle_{\tau}$ and the duality between elements of
$\Wts^{-1,q'}(\Omega)$ and $\Wts^{1,q}(\Omega)$ is denoted by
$\langle\, .\, ,\, .\, \rangle_{\tau,\sigma}$.

\item
$\Wts^{1,q}(\Omega)^{\perp}$ denotes the space of annihilators of
$\Wts^{1,q}(\Omega)$ in $\Wt^{-1,q'}(\Omega)$. i.e.~the space
$\bigl\{\bfg\in\Wt^{-1,q'}(\Omega)$;
$\forall\bfvarphi\in\Wts^{1,q}(\Omega):
\langle\bfg,\bfvarphi\rangle_{\tau}=0\bigr\}$.

\end{list}

\noindent
{\bf 2.2. $\bfL^{q'}(\Omega)$ and $\Ls^{q'}(\Omega)$ as subspaces
of $\Wt^{-1,q'}(\Omega)$ and $\Wts^{-1,q'}(\Omega)$,
respectively.} \ The Lebesgue space $\bfL^{q'}(\Omega)$ can be
identified with a subspace of $\Wt^{-1,q'}(\Omega)$ so that if
$\bfg\in\bfL^{q'}(\Omega)$ then
\begin{equation}
\langle\bfg,\bfvarphi\rangle_{\tau}\ :=\
\int_{\Omega}\bfg\cdot\bfvarphi\; \rmd\bfx \label{2.1}
\end{equation}
for all $\bfvarphi\in \Wt^{1,q}(\Omega)$. Similarly,
$\Ls^{q'}(\Omega)$ can be identified with a subspace of
$\Wts^{-1,q'}(\Omega)$ so that if $\bfg\in\Ls^{q'}(\Omega)$ then
\begin{equation}
\langle\bfg,\bfvarphi\rangle_{\tau,\sigma}\ :=\
\int_{\Omega}\bff\cdot\bfvarphi\; \rmd\bfx \label{2.2}
\end{equation}
for all $\bfvarphi\in\Wts^{1,q}(\Omega)$. Thus, if
$\bfg\in\Ls^{q'}(\Omega)$ and $\bfvarphi\in\Wts^{1,q}(\Omega)$
then the dualities $\langle\bfg,\bfvarphi\rangle_{\tau}$ and
$\langle\bfg,\bfvarphi\rangle_{\tau,\sigma}$ coincide.

Note that if $\, \bfg\in\bfL^{q'}(\Omega)$ then the integral on
the right hand side of (\ref{2.1}) also defines a bounded linear
functional on $\Wts^{1,q}(\Omega)$. This, however, does not mean
that $\bfL^{q'}(\Omega)$ can be identified with a subspace of
$\Wts^{-1,q'}(\Omega)$. The reason is, for instance, that the
spaces $\bfL^{q'}(\Omega)$ and $\Wts^{-1,q'}(\Omega)$ do not have
the same zero element. (If $\psi$ is a non-constant function in
$C^{\infty}_0(\Omega)$ then $\nabla\psi$ is a non-zero element of
$\bfL^{q'}(\Omega)$, but it induces the zero element of
$\Wts^{-1,q'}(\Omega)$.)

\vspace{4pt} \noindent
{\bf 2.3. Definition and some properties of operator $\cPs{q'}$.}
\ $\Wts^{1,q}(\Omega)$ is a closed subspace of
$\Wt^{1,q}(\Omega)$. If $\bfg\in\Wt^{-1,q'}(\Omega)$ (i.e.~$\bff$
is a bounded linear functional on $\Wt^{1,q}(\Omega)$) then we
denote by $\cPs{q'}\bff$ the element of $\Wts^{-1,q'}(\Omega)$,
defined by the equation
\begin{displaymath}
\langle \cPs{q'}\bfg,\bfvarphi\rangle_{\tau,\sigma}\ :=\
\langle\bfg,\bfvarphi\rangle_{\tau} \qquad \mbox{for all}\
\bfvarphi\in\Wts^{1,q}(\Omega).
\end{displaymath}
Obviously, $\cPs{q'}$ is a linear operator from
$\Wt^{-1,q'}(\Omega)$ to $\Wts^{-1,q'}(\Omega)$, whose domain is
the whole space $\Wt^{-1,q'}(\Omega)$.

\begin{lemma} \label{L2.1}
The operator $\cPs{q'}$ is bounded, its range is
$\Wts^{-1,q'}(\Omega)$ and $\cPs{q'}$ is not one-to-one.
\end{lemma}

\begin{proof} \rm
The boundedness of operator $\cPs{q'}$ directly follows from the
definition of the norms in the spaces $\Wt^{-1,q'}(\Omega)$,
$\Wts^{-1,q'}(\Omega)$ and the definition of $\cPs{q'}$.

Let $\bfg\in\Wts^{-1,q'}(\Omega)$. There exists (by the
Hahn-Banach theorem) an extension of $\bfg$ from
$\Wts^{1,q}(\Omega)$ to $\Wt^{1,q}(\Omega)$, which we denote by
$\widetilde{\bfg}$. The extension is an element of
$\Wt^{-1,q'}(\Omega)$, satisfying
$\|\widetilde{\bfg}\|_{-1,q'}=\|\bfg\|_{-1,q';\, \sigma}$ and
\begin{displaymath}
\langle \widetilde{\bfg},\bfvarphi\rangle_{\tau}\ =\
\langle\bfg,\bfvarphi\rangle_{\tau,\sigma}
\end{displaymath}
for all $\bfvarphi\in\Wts^{1,q}(\Omega)$. This shows that
$\bfg=\cPs{q'}\widetilde{\bfg}$. Consequently, the range of
$\cPs{q'}$ is the whole space $\bfW^{-1,q'}_{0,\sigma}(\Omega)$.

Finally, considering $\bfg=\nabla\psi$ for $\psi\in
C^{\infty}_0(\Omega)$, we get
\begin{displaymath}
\langle \cPs{q'}\bfg,\bfvarphi\rangle_{\tau,\sigma}\ =\ \langle
\bfg,\bfvarphi \rangle_{\tau}\ =\ \int_{\Omega}\nabla
\psi\cdot\bfvarphi\; \rmd\bfx\ =\ 0
\end{displaymath}
for all $\bfvarphi\in\Wts^{1,q}(\Omega)$. This shows that the
operator $\cPs{q'}$ is not one-to-one.
\end{proof}

\vspace{4pt} \noindent
{\bf 2.4. The relation between operator $\cPs{q'}$ and the
Helmholtz projection.} \ If each function
$\bfg\in\bfL^{q'}(\Omega)$ can be uniquely expressed in the form
$\bfg=\bfv+\nabla\psi$ for some $\bfv\in\Ls^{q'}(\Omega)$ and
$\nabla\psi\in\bfG_{q'}(\Omega)$, which is equivalent to the
validity of the decomposition
\begin{equation}
\bfL^{q'}(\Omega)\ =\ \Ls^{q'}(\Omega)\oplus\bfG_{q'}(\Omega),
\label{2.1*}
\end{equation}
then we write $\bfv=\Ps{q'}\bfg$. Decomposition (\ref{2.1*}) is
called the {\it Helmholtz decomposition} and the operator
$\Ps{q'}$ is called the {\it Helmholtz projection.} The existence
of the Helmholtz decomposition depends on exponent $q'$ and the
shape of domain $\Omega$. If $q'=2$ then the Helmholtz
decomposition exists on an arbitrary domain $\Omega$ and $\Ps{2}$,
respectively $I-\Ps{2}$, is an orthogonal projection of
$\bfL^2(\Omega)$ onto $\Ls^2(\Omega)$, respectively onto
$\bfG_2(\Omega)$. (See e.g.~\cite{Ga1}.) If $q'\not=2$ then
various sufficient conditions for the existence of the Helmholtz
decomposition can be found e.g.~in \cite{FaKoSo}, \cite{FuMo},
\cite{Ga1}, \cite{GeShe}, \cite{KoYa} and \cite{SiSo}.

Further on in this paragraph, we assume that the Helmholtz
decomposition of $\bfL^{q'}(\Omega)$ exists. Let
$\bfg\in\bfL^{q'}(\Omega)$. Treating $\bfg$ as an element of
$\Wt^{-1,q'}(\Omega)$ in the sense of paragraph 2.2, we have
$\langle\cPs{q'}\bfg,\bfvarphi\rangle_{\tau,\sigma}= \langle\bfg,
\bfvarphi\rangle_{\tau}$ for all $\bfvarphi\in\Wts^{1,q}(\Omega)$.
Writing $\bfg=\Ps{q'}\bfg+(I-\Ps{q'})\bfg$, we also have
\begin{displaymath}
\langle\bfg, \bfvarphi\rangle_{\tau}\ =\
\blangle\Ps{q'}\bfg+(I-\Ps{q'})\bfg,\bfvarphi\brangle_{\tau}\ =\
\blangle\Ps{q'}\bfg,\bfvarphi\brangle_{\tau}
\end{displaymath}
for all $\bfvarphi\in\Wts^{1,q}(\Omega)$, because
$(I-\Ps{q'})\bfg\in\bfG_{q'}(\Omega)$. Furthermore,
\begin{displaymath}
\blangle\Ps{q'}\bfg,\bfvarphi\brangle_{\tau}\ =\
\blangle\Ps{q'}\bfg,\bfvarphi\brangle_{\tau,\sigma},
\end{displaymath}
because $\Ps{q'}\bfg\in\Ls^{q'}(\Omega)$,
$\bfvarphi\in\Wts^{1,q}(\Omega)$ and the formulas (\ref{2.1}) and
(\ref{2.2}) show that the dualities
$\langle\Ps{q'}\bfg,\bfvarphi\rangle_{\tau}$ and
$\langle\Ps{q'}\bfg,\bfvarphi\rangle_{\tau,\sigma}$ are expressed
by the same integrals. Hence
$\langle\cPs{q'}\bfg,\bfvarphi\rangle_{\tau,\sigma}$ coincides
with $\langle\Ps{q'}\bfg,\bfvarphi\rangle_{\tau,\sigma}$ for all
$\bfvarphi\in\Wts^{1,q}(\Omega)$. Consequently, $\cPs{q'}\bfg$ and
$\Ps{q'}\bfg$ represent the same element of
$\Wts^{-1,q'}(\Omega)$. As $\Ps{q'}\bfg\in\Ls^{q'}(\Omega)$,
$\cPs{q'}\bfg$ can also be considered to be an element of
$\Ls^{q'}(\Omega)$, which induces a functional in
$\Wts^{-1,q}(\Omega)$ in the sense of paragraph 2.2. Thus, {\it
the Helmholtz projection $\Ps{q'}$ coincides with the restriction
of $\cPs{q'}$ to $\bfL^{q'}(\Omega)$.}

\vspace{4pt} \noindent
{\bf 2.5. More on the space $\Wts^{1,q}(\Omega)^{\perp}$.} \
Identifying $\bfG_{q'}(\Omega)$ with a subspace of
$\Wt^{-1,q'}(\Omega)$ in the sense of paragraph 2.2, {\it we
denote by ${}^{\perp}\bfG_{q'}(\Omega)$ the linear space
$\bigl\{\bfvarphi\in\Wt^{1,q}(\Omega)$;
$\forall\br\bfg\in\bfG_{q'}(\Omega):
\langle\bfg,\bfvarphi\rangle_{\tau}=0\bigr\}$.} Using \cite[Lemma
III.2.1]{Ga1}, we deduce that
$\Wts^{1,q}(\Omega)={}^{\perp}\bfG_{q'}(\Omega)$. Hence
$\Wts^{1,q}(\Omega)^{\perp}=(^{\perp}\bfG_{q'}(\Omega))^{\perp}$
and applying Theorem 4.7 in \cite{Ru}, we observe that
$\Wts^{1,q}(\Omega)^{\perp}$ is a closure of $\bfG_{q'}(\Omega)$
in the weak-$*$ topology of $\Wt^{-1,q'}(\Omega)$. The next lemma
tells us more on elements of $\Wts^{1,q}(\Omega)^{\perp}$.

\begin{lemma} \label{L2.2}
Let $\bfF\in\Wts^{1,q}(\Omega)^{\perp}$ and
$\Omega_0\subset\subset\Omega$ be a nonempty sub-domain of
$\Omega$. Then there exists a unique $p\in L^{q'}_{loc}(\Omega)$
such that $p\in L^{q'}(\Omega_R)$ for all $R>0$, \
$\int_{\Omega_0}p\; \rmd\bfx=0$ \ and
\begin{alignat}{5}
& \|p\|_{q';\, \Omega_R}\ &&\leq\ c(R)\, \|\bfF\|_{-1,q} \quad &&
\mbox{for all}\ R>0, \label{2.9*} \\
& \blangle\bfF,\bfpsi\brangle_{\tau}\ &&=\ -\int_{\Omega}p\
\div\bfpsi\; \rmd\bfx \quad && \mbox{for all}\
\bfpsi\in\Wtc^{1,q}(\Omega). \label{2.7*}
\end{alignat}
\end{lemma}

\begin{proof}
Let $\{\Omega_n\}$ be the sequence of domains from condition (i).
We can assume without the loss of generality that
$\Omega_0\subseteq\Omega_1$. Let $n\in\N$. Denote by $L^q_{\rm mv=
0}(\Omega_n)$ the space of all functions from $L^q(\Omega_n)$,
whose mean value in $\Omega_n$ is zero. There exists a bounded
linear operator $\gB: L^q_{\rm mv=0}
(\Omega_n)\to\bfW^{1,q}_0(\Omega_n)$, such that
\begin{displaymath}
\div\gB(g)\ =\ g
\end{displaymath}
for all $g\in L^q_{\rm mv=0}(\Omega_n)$. Operator $\gB$ is often
called the {\it Bogovskij} or {\it Bogovskij--Pileckas} operator.
More information on operator $\gB$, including its construction,
can be found e.g.~in \cite[Sec.~III.3]{Ga1} or in \cite{BoSo}.

Denote by $\Wt^{1,q}(\Omega)_n$, respectively
$\Wts^{1,q}(\Omega)_n$, the space of all functions from
$\Wt^{1,q}(\Omega)$, respectively from $\Wts^{1,q}(\Omega)$, that
have a support in $\overline{\Omega_n}$. Let
$\bfpsi\in\Wt^{1,q}(\Omega)_n$. Then the restriction of
$\div\bfpsi$ to $\Omega_n$ (which we again denote by $\div\bfpsi$
in order to keep a simple notation) belongs to $L^q_{\rm
mv=0}(\Omega_n)$ and $\gB(\div\bfpsi_n)\in\bfW^{1,q}_0(\Omega_n)$.
Identifying $\gB(\div\bfpsi)$ with a function from
$\bfW^{1,q}_0(\Omega)$ that equals zero in
$\Omega\smallsetminus\Omega_n$, we have
\begin{displaymath}
\bfpsi\ =\ \gB(\div\bfpsi)+\bfw,
\end{displaymath}
where $\bfw$ is an element of $\Wts^{1,q}(\Omega)$, satisfying
$\bfw=\bfpsi=\bfzero$ in $\Omega\smallsetminus\Omega_n$. Hence
\begin{equation}
\blangle\bfF,\bfpsi\brangle_{\tau}\ =\ \blangle\bfF,\gB(\div\bfpsi
\brangle_{\tau}. \label{2.2*}
\end{equation}
As $\bfF$ is a bounded linear functional on $\Wt^{1,q}(\Omega)$,
vanishing on the subspace $\Wts^{1,q}(\Omega)$, its restriction to
$\Wt^{1,q}(\Omega)_n$ is an element of $\Wt^{-1,q'}(\Omega)_n$,
vanishing on $\Wts^{1,q}(\Omega)_n$. Furthermore, identifying
functions from $\Wt^{1,q}(\Omega)_n$ with their restrictions to
$\Omega_n$, we can also consider $\bfF$ to be an element of
$\bfW^{-1,q'}_0(\Omega_n)$, vanishing on
$\bfW^{1,q}_{0,\sigma}(\Omega_n)$. Thus, due to Lemma 1.4 in
\cite{Ne2}, there exists $c(n)>0$ and a unique function $p_n\in
L^{q'}(\Omega_n)$ such that $\int_{\Omega_0}p_n\; \rmd\bfx=0$ and
\begin{align}
\|p_n\|_{q';\, \Omega_n}\ &\leq\ c(n)\, \|\bfF\|_{-1,q;\,
\Omega_n}\ \leq\ c(n)\, \|\bfF\|_{-1,q}, \label{2.8*} \\
\blangle\bfF,\bfzeta\brangle_{\Omega_n}\ &=\ -\int_{\Omega_n}p_n\
\div\bfzeta\; \rmd\bfx \label{2.3*}
\end{align}
for all $\bfzeta\in\bfW^{1,q}_0(\Omega_n)$. Using identity
(\ref{2.3*}) with $\bfzeta=\gB(\div\bfpsi)$, we obtain
\begin{displaymath}
\blangle\bfF,\gB(\div\bfpsi)\brangle_{\tau}\ \equiv\
\blangle\bfF,\gB(\div\bfpsi)\brangle_{\Omega_n}\ =\
-\int_{\Omega_n}\! p_n\ \div\gB(\div\bfpsi)\; \rmd\bfx\ =\
-\int_{\Omega_n}\! p_n\ \div\bfpsi\; \rmd\bfx.
\end{displaymath}
As the same identities also hold for $n+1$ instead of $n$, we
deduce that $p_{n+1}=p_n$ in $\Omega_n$. Hence we may define
function $p$ in $\Omega$ by the formula $p:=p_n$ in $\Omega_n$ and
we have
\begin{equation}
\blangle\bfF,\gB(\div\bfpsi)\brangle_{\tau}\ =\ -\int_{\Omega} p\
\div\bfpsi\; \rmd\bfx. \label{2.4*}
\end{equation}
If $\bfpsi\in\Wtc^{1,q}(\Omega)$ then
$\bfpsi\in\Wt^{1,q}(\Omega)_n$ for sufficiently large $n$ and
(\ref{2.4*}) holds as well. Inequality (\ref{2.9*}) now follows
from (\ref{2.8*}). Identities (\ref{2.2*}) and (\ref{2.4*}) imply
(\ref{2.7*}).
\end{proof}

\vspace{4pt}
Note that if $\Omega$ is a bounded Lipschitz domain then the
choice $\Omega_0=\Omega$ is also possible in Lemma \ref{L2.2}.


\section{Three equivalent weak formulations of the Navier--Stokes
initial-boundary value problem (\ref{1.1})--(\ref{1.4})}
\label{S3}

Recall that $\Omega$ is supposed to be a locally Lipschitz domain
in $\R^3$.

\vspace{4pt} \noindent
{\bf 3.1. The 1st weak formulation of the Navier--Stokes IBVP
(\ref{1.1})--(\ref{1.4}).} \ {\it Given $\bfu_0\in\Ls^2(\Omega)$
and $\bff\in L^2(0,T$; $\Wt^{-1,2}(\Omega))$. A function $\,
\bfu\in L^{\infty}(0,T;\ \Ls^2(\Omega)) \cap L^2(0,T;\
\Wts^{1,2}(\Omega))$ is said to be a weak solution to the problem
(\ref{1.1})--(\ref{1.4}) if the trace of $\bfu$ on $\Gamma_T$ is
in $L^2(0,T$; $\bfL^2(\partial\Omega))$ and $\bfu$ satisfies
\begin{align}
\int_0^T & \int_{\Omega}\bigl[-\partial_t\bfphi\cdot\bfu+
\bfu\cdot\nabla\bfu\cdot\bfphi+2\nu\br(\nabla\bfu)_s:
(\nabla\bfphi)_s\bigr]\, \rmd\bfx\, \rmd t \nonumber \\
& +\int_0^T\int_{\partial\Omega}\gamma\br\bfu\cdot\bfphi\; \rmd
S\, \rmd t\, =\, \int_0^T\blangle\bff,\bfphi\brangle_{\tau}\; \rmd
t+\int_{\Omega}\bfu_0\cdot\bfphi(.\, ,0)\, \rmd\bfx
\label{3.1}
\end{align}
for all vector--functions $\bfphi\in C^{\infty}_0\bigl([0,T);\;
\Wtsc^{1,2}(\Omega)\bigr)$.}

\vspace{4pt}
Equation (\ref{3.1}) follows from (\ref{1.1}), (\ref{1.2}) if one
formally multiplies equation (\ref{1.1}) by the test function
$\bfphi\in C^{\infty}_0\bigl([0,T);\; \Wtsc^{1,2} (\Omega)\bigr)$,
applies the integration by parts and uses the boundary conditions
(\ref{1.3}) and the initial condition (\ref{1.4}). As the integral
of $\nabla p\cdot\bfphi$ vanishes, the pressure $p$ does not
explicitly appear in (\ref{3.1}).

On the other hand, if $\bff\in\bfL^2(Q_T)$ and $\bfu$ is a weak
solution with the additional properties
$\partial_t\bfu\in\bfL^2(Q_T)$ and $\bfu\in L^2(0,T;\,
\bfW^{2,2}(\Omega))$ then, considering the test functions $\bfphi$
in (\ref{3.1}) of the form $\bfphi(\bfx,t)=\bfvarphi(\bfx)\,
\vartheta(t)$ where $\bfvarphi\in\Wtsc^{1,2}(\Omega)$ and
$\vartheta\in C^{\infty}_0((0,T))$, and applying the backward
integration by parts, one obtains the equation
\begin{displaymath}
\int_{\Omega} \bigl( \partial_t\bfu+\bfu\cdot\nabla\bfu-
\nu\Delta\bfu-\bff \bigr)\cdot\bfvarphi\; \rmd\bfx\ =\ 0
\end{displaymath}
for a.a.~$t\in(0,T)$. As $\Wtsc^{1,2}(\Omega)$ is dense in
$\Ls^2(\Omega)$, this equation shows that $\Ps{2}[\partial_t\bfu
+\bfu\cdot\nabla\bfu-\nu\Delta\bfu- \bff]=\bfzero$ at a.a.~time
instants $t\in(0,T)$. Consequently, to a.a.~$t\in(0,T)$, there
exists $p\in W^{1,2}_{\rm loc}(\Omega)$ such that $\nabla
p=(I-\Ps{2})[\partial_t\bfu+ \bfu\cdot\nabla\bfu-\nu\Delta\bfu-
\bff]$ and the functions $\bfu$ and $p$ satisfy equation
(\ref{1.1}) (as an equation in $\bfL^2(\Omega)$) at a.a.~time
instants $t\in(0,T)$. It follows from the boundedness of
projection $\Ps{2}$ in $\bfL^2(\Omega)$ and the assumed properties
of functions $\bfu$ and $\bff$ that $\nabla p\in \bfL^2(Q_T)$.
Considering afterwards the test functions $\bfphi$ as in
(\ref{3.1}), and integrating by parts in (\ref{3.1}), we get
\begin{displaymath}
\int_0^T\int_{\Omega}\bigl( \partial_t\bfu+\bfu\cdot\nabla\bfu-
\nu\Delta\bfu-\bff \bigr)\cdot\bfphi\; \rmd\bfx+
\int_0^T\int_{\partial\Omega} \bigl( [\bbTd(\bfu)\cdot\bfn]+
\gamma\bfu \bigr)\cdot\bfphi\; \rmd S\, \rmd t\ =\ 0
\end{displaymath}
The first integral is equal to zero, because the expression in the
parentheses equals $-\nabla p$ a.e.~in $Q_T$ and the integral
$\nabla p\cdot\bfphi$ in $\Omega$ equals zero for
a.a.~$t\in(0,T)$. In the second integral, since both $\bfu(\, .\,
,t)$ and  $\bfphi(\, .\, ,t)$ are tangent on $\partial\Omega$, we
can replace $[\bbTd(\bfu)\cdot\bfn]+\gamma\bfu$ by
$[\bbTd(\bfu)\cdot\bfn]_{\tau}+\gamma\bfu$ and we thus obtain
\begin{displaymath}
\int_0^T\int_{\partial\Omega} \bigl(
[\bbTd(\bfu)\cdot\bfn]_{\tau}+\gamma\bfu \bigr)\cdot\bfphi\; \rmd
S\, \rmd t\ =\ 0.
\end{displaymath}
As this equation holds for all test functions $\bfphi\in
C^{\infty}_0\bigl([0,T);\; \Wtsc^{1,2}(\Omega)\bigr)$, we deduce
that $\bfu$ satisfies the boundary condition (\ref{1.3}b). Recall
that this procedure works only under additional assumptions on
smoothness of the weak solution $\bfu$ and function $\bff$. On a
general level, however, it is not known whether the existing weak
solution is smooth. Nevertheless, we show in subsection 4.4 that
there exists a certain pressure, which can be naturally associated
with the weak solution to (\ref{1.1})--(\ref{1.4}). The pressure
generally exists only as a distribution, see Theorem \ref{T4.2}.

\vspace{6pt} \noindent
{\bf 3.2. The 2nd weak formulation of the Navier-Stokes IBVP
(\ref{1.1})--(\ref{1.4}).} \ We define the operators
$\cA:\Wt^{1,2}(\Omega)\to\Wt^{-1,2}(\Omega)$ and
$\cB:\bigl[\Wt^{1,2}(\Omega)\bigr]^2\to \Wt^{-1,2}(\Omega)$ by the
equations
\begin{align*}
& \blangle\cA\bfv,\bfvarphi\brangle_{\tau}\ :=\ \int_{\Omega}
2\nu\br(\nabla\bfv)_s:(\nabla\bfvarphi)_s\;
\rmd\bfx+\int_{\partial\Omega}\gamma\bfv\cdot\bfvarphi\; \rmd S &&
\mbox{for}\ \bfv,\bfvarphi\in\Wt^{1,2}(\Omega), \\
& \blangle\cB(\bfv,\bfw),\bfvarphi\brangle_{\tau}\ :=\
\int_{\Omega} \bfv\cdot\nabla\bfw\cdot\bfvarphi\; \rmd\bfx &&
\mbox{for}\ \bfv,\bfw,\bfvarphi\in\Wt^{1,2}(\Omega).
\end{align*}
By Korn's inequality (see e.g.~\cite[Lemma 4]{SoSc}) and
inequality \cite[(II.4.5), p.~63]{Ga1}, we have $\cn01$
\begin{equation}
\blangle\cA\bfv,\bfv\brangle_{\tau}\ =\ \int_{\Omega}\nu\,
|(\nabla\bfv)_s|^2\; \rmd\bfx+ \int_{\partial\Omega}\gamma\,
|\bfv|^2\; \rmd S\ \geq\ \cc01\br \nu\, \|\nabla\bfv\|_2^2.
\label{3.2}
\end{equation}
Furthermore, using the boundedness of the operator of traces from
$\Wt^{1,2}(\Omega)$ to $\bfL^2(\partial\Omega)$, we can also
deduce that there exists $\cn02\cc02>0$ such that
\begin{equation}
\|\cA\bfv\|_{-1,2}\ \leq\ \cc02\, \|\nabla\bfv\|_2
\label{3.3}
\end{equation}
for all $\bfv\in\Wt^{1,2}(\Omega)$. Thus, $\cA$ is a bounded
one--to--one operator, mapping $\Wt^{1,2}(\Omega)$ into
$\Wt^{-1,2}(\Omega)$. If $k>0$ then the range of $\cA+kI$ is the
whole space $\Wt^{-1,2}(\Omega)$ (by the Lax--Milgram theorem) and
$(\cA+kI)^{-1}$ is a bounded operator from $\Wt^{-1,2}(\Omega)$
onto $\Wt^{1,2}(\Omega)$. If $\Omega$ is bounded then the same
statements also hold for $k=0$. The bilinear operator $\cB$
satisfies
\begin{align}
& \|\cB(\bfv,\bfw)\|_{-1,2}\ =\ \sup_{\boldsymbol{\varphi}
\in\Wt^{1,2}(\Omega),\ \boldsymbol{\varphi}\not=\bfzero}
\frac{|\br\langle \cB(\bfv,\bfw),\bfvarphi\rangle_{\tau}\br|}
{\|\bfvarphi\|_{1,2}} \nonumber \\
& \hspace{6pt} =\ \sup_{\boldsymbol{\varphi}\in
\Wt^{1,2}(\Omega),\ \boldsymbol{\varphi}\not=\bfzero}
\frac{|(\bfv\cdot\nabla\bfw,\,\bfvarphi)_2|}
{\|\bfvarphi\|_{1,2}}\ \leq \sup_{\boldsymbol{\varphi}\in
\Wt^{1,2}(\Omega),\ \boldsymbol{\varphi}\not=\bfzero}
\frac{\|\bfv\|_2^{1/2}\, \|\bfv\|_6^{1/2}\, \|\nabla\bfw\|_2\,
\|\bfvarphi\|_6}{\|\bfvarphi\|_{1,2}} \nonumber \\
\noalign{\vskip 4pt}
& \hspace{6pt} \leq\ c\, \|\bfv\|_2^{1/2}\,
\|\nabla\bfv\|_2^{1/2}\, \|\nabla\bfw\|_2. \label{3.4}
\end{align}
(We have used the imbedding inequality $\|\bfv\|_6\leq c\,
\|\bfv\|_{1,2}$. Here and further on, $c$ denotes the generic
constant.)

Let $\bfu$ be a weak solution of the IBVP (\ref{1.1})--(\ref{1.4})
in the sense of paragraph 3.1. It follows from the estimates
(\ref{3.3}) and (\ref{3.4}) that
\begin{equation}
\cA\bfu\in L^2(0,T;\, \Wt^{-1,2}(\Omega)) \quad \mbox{and} \quad
\cB(\bfu,\bfu)\in L^{4/3}(0,T;\, \Wt^{-1,2}(\Omega)).
\label{3.5}
\end{equation}
Considering $\bfphi$ in (\ref{3.1}) in the form
$\bfphi(\bfx,t)=\bfvarphi(\bfx)\, \vartheta(t)$, where
$\bfvarphi\in\Wtsc^{1,2}(\Omega)$ and $\vartheta\in
C^{\infty}_0((0,T))$, we deduce that $\bfu$ satisfies the equation
\begin{equation}
\frac{\rmd}{\rmd\br t}\, (\bfu,\bfvarphi)_2+\blangle
\cA\bfu,\bfvarphi \brangle_{\tau}+\blangle \cB(\bfu,\bfu),
\bfvarphi\brangle_{\tau}\ =\ \langle \bff,\bfvarphi \rangle_{\tau}
\label{3.6}
\end{equation}
a.e.~in\ $(0,T)$, where the derivative of $(\bfu,\bfvarphi)_2$
means the derivative in the sense of distributions. As the space
$\Wtsc^{1,2}(\Omega)$ is dense in $\Wts^{1,2}(\Omega)$,
(\ref{3.6}) holds for all $\bfvarphi\in\Wts^{1,2}(\Omega)$. It
follows from (\ref{3.5}) that $\langle\cA\bfu,\bfvarphi
\rangle_{\tau} \in L^2(0,T)$ and $\langle\cB(\bfu,\bfu),\bfvarphi
\rangle_{\tau} \in L^{4/3}(0,T)$. Since
$\langle\bff,\bfvarphi\rangle_{\tau}\in L^2(0,T)$, we obtain from
(\ref{3.6}) that the distributional derivative of
$(\bfu,\bfvarphi)_2$ with respect to $t$ is in $L^{4/3}(0,T)$.
Hence $(\bfu,\bfvarphi)_2$ is a.e.~in $[0,T)$ equal to a
continuous function and the weak solution $\bfu$ is (after a
possible redefinition on a set of measure zero) a weakly
continuous function from $[0,T)$ to $\Ls^2(\Omega)$. Now, one can
easily deduce from (\ref{3.1}) that $\bfu$ satisfies the initial
condition (\ref{1.4}) in the sense that
\begin{equation}
(\bfu,\bfvarphi)_2\br\bigl|_{t=0}\ =\ (\bfu_0,\bfvarphi)_2
\label{3.7}
\end{equation}
for all $\bfvarphi\in\Wts^{1,2}(\Omega)$. Thus, we come to the 2nd
weak formulation of the IBVP (\ref{1.1})--(\ref{1.4}):

\vspace{4pt}
{\it Given $\bfu_0\in\Ls^2(\Omega)$ and $\bff\in L^2(0,T;\
\Wt^{-1,2}(\Omega))$. Find $\bfu\in L^{\infty}(0,T;\
\Ls^2(\Omega))\cap L^2(0,T$; $\Wts^{1,2}(\Omega))$ (called the
weak solution) such that $\bfu$ satisfies equation (\ref{3.6})
a.e.~in $(0,T)$ and the initial condition (\ref{3.7}) for all
$\bfvarphi\in\Wts^{1,2}(\Omega)$.}

\vspace{4pt}
We have shown that if $\bfu$ is a weak solution of the IBVP
(\ref{1.1})--(\ref{1.4}) in the sense of the 1st definition (see
paragraph 3.1) then it also satisfies the 2nd definition. Applying
standard arguments, one can also show the opposite, i.e.~if $\bfu$
satisfies the 2nd definition then it also satisfies the 1st
definition.

\vspace{5pt} \noindent
{\bf 3.3. The 3rd weak formulation of the Navier-Stokes IBVP
(\ref{1.1})--(\ref{1.4}).} \ Equation (\ref{3.6}) can also be
written in the equivalent form
\begin{equation}
\frac{\rmd}{\rmd\br t}\, (\bfu,\bfvarphi)_2+\blangle
\cPs{2}\br\cA\bfu,\bfvarphi \brangle_{\tau,\sigma}+
\blangle\cPs{2}\br\cB(\bfu,\bfu),\bfvarphi\brangle_{\Omega,
\sigma}\ =\ \blangle \cPs{2}\br\bff,
\bfvarphi\brangle_{\tau,\sigma}. \label{3.8}
\end{equation}
Let us denote by $(\bfu')_{\sigma}$ the distributional derivative
with respect to $t$ of $\bfu$, as a function from $(0,T)$ to
$\Wts^{-1,2}(\Omega)$. (We explain later why we use the notation
$(\bfu')_{\sigma}$ and not just $\bfu'$.) Equation (\ref{3.8}) can
also be written in the form
\begin{equation}
(\bfu')_{\sigma}+\cPs{2}\br\cA\bfu+\cPs{2}\br\cB(\bfu,\bfu)\ =\
\cPs{2}\br\bff, \label{3.9}
\end{equation}
which is an equation in $\Wts^{-1,2}(\Omega)$, satisfied a.e.~in
the time interval $(0,T)$. (This can be deduced by means of Lemma
III.1.1 in \cite{Te}.) Due to (\ref{3.5}) and (\ref{3.6}),
$(\bfu')_{\sigma}\in L^{4/3}(0,T;\, \Wts^{-1,2}(\Omega))$. Hence
$\bfu$ coincides a.e.~in $(0,T)$ with a continuous function from
$[0,T)$ to $\Wts^{-1,2}(\Omega)$ and it is therefore meaningful to
prescribe an initial condition for $\bfu$ at time $t=0$. Thus, we
obtain the 3rd equivalent definition of a weak solution to the
IBVP (\ref{1.1})--(\ref{1.4}):

\vspace{4pt}
{\it Given $\bfu_0\in\Ls^2(\Omega)$ and $\bff\in L^2(0,T$;
$\Wt^{-1,2}(\Omega))$. Function $\bfu\in L^{\infty}(0,T$;
$\Ls^2(\Omega))\cap L^2(0,T$; $\Wts^{1,2}(\Omega))$ is called a
weak solution to the IBVP (\ref{1.1})--(\ref{1.4}) if $\bfu$
satisfies equation (\ref{3.9}) a.e.~in the interval $(0,T)$ and
the initial condition (\ref{1.4}).}

\vspace{4pt}
We have explained that if $\bfu$ is a weak solution in the sense
of the 2nd definition then it satisfies the 3rd definition. The
validity of the opposite implication can be again verified by
means of Lemma III.1.1 in \cite{Te}.

\vspace{4pt} \noindent
{\bf 3.4. Remark.} \ \rm Recall that $(\bfu')_{\sigma}$ is the
distributional derivative with respect to $t$ of $\bfu$, as a
function from $(0,T)$ to $\Wts^{-1,2}(\Omega)$. It is not the same
as the distributional derivative with respect to $t$ of $\bfu$, as
a function from $(0,T)$ to $\Wt^{-1,2}(\Omega)$, which can be
naturally denoted by $\bfu'$. As it is important to distinguish
between these two derivatives, we use the different notation. We
can formally write $(\bfu')_{\sigma}=\cPs{2}\bfu'$.

Since $(\bfu')_{\sigma}\in L^{4/3}(0,T;\, \Wts^{-1,2}(\Omega))$,
$\bfu$ coincides a.e.~in $(0,T)$ with a continuous function from
$[0,T)$ to $\Wts^{-1,2}(\Omega)$. According to what is said in the
first part of this remark, this, however, does not imply that
$\bfu$ coincides a.e.~in $(0,T)$ with a continuous function from
$[0,T)$ to $\Wt^{-1,2}(\Omega)$.


\section{An associated pressure, its uniqueness and existence}
\label{S4}

{\bf 4.1. An associated pressure.} \ {\it Let $\bfu$ be a weak
solution to the IBVP (\ref{1.1})--(\ref{1.4}). A distribution $p$
in $Q_T$ is called an associated pressure if the pair $(\bfu,p)$
satisfies the equations (\ref{1.1}), (\ref{1.2}) in the sense of
distributions in $Q_T$.}

\vspace{4pt} \noindent
{\bf 4.2. On uniqueness of the associated pressure.} \ Let $\bfu$
be a weak solution to the IBVP (\ref{1.1})--(\ref{1.4}) and $p$ be
an associated pressure.

If $G$ is a distribution in $(0,T)$ and $\psi\in
C_0^{\infty}(Q_T)$ then we define a distribution $g$ in $Q_T$ by
the formula
\begin{equation}
\bllangle g,\psi\brrangle_{Q_T}\ :=\ \Blangle G,\
\int_{\Omega}\psi\; \rmd\bfx \Brangle_{(0,T)}, \label{4.11}
\end{equation}
where $\llangle\, .\, ,\, .\, \rrangle_{Q_T}$, respectively
$\langle\, .\, ,\, .\, \rangle_{(0,T)}$, denotes the action of a
distribution in $Q_T$ on a function from $C^{\infty}_0(Q_T)$ or
$\bfC^{\infty}_0(Q_T)$, respectively the action of a distribution
in $(0,T)$ on a function from $C_0^{\infty}((0,T))$. Obviously, if
$\bfphi\in C^{\infty}_0\bigl((0,T);\, \Wtc^{1,2}(\Omega)\bigr)$
then
\begin{equation}
\bllangle\nabla g,\bfphi\brrangle_{Q_T}\ =\ -\bllangle
g,\div\bfphi\brrangle_{Q_T}\ =\ -\Blangle G,\
\int_{\Omega}\div\bfphi\; \rmd\bfx\Brangle_{(0,T)}\ =\ 0,
\label{4.12}
\end{equation}
because $\int_{\Omega}\div\bfphi(\, .\, ,t)\; \rmd\bfx=0$ for all
$t\in(0,T)$. Thus, $p+g$ is a pressure, associated with the weak
solution $\bfu$ to the IBVP (\ref{1.1})--(\ref{1.4}), too.

For $h\in C_0^{\infty}((0,T))$, define
\begin{equation}
\blangle G,h \brangle_{(0,T)}\ :=\ \bllangle
g,\psi\brrangle_{Q_T}, \label{4.13}
\end{equation}
where $\psi\in C_0^{\infty}(Q_T)$ is chosen so that
$h(t)=\int_{\Omega}\psi(\bfx,t)\; \rmd\bfx$ for all $t\in(0,T)$.
The definition of the distribution $G$ is independent of the
concrete choice of function $\psi$ due to these reasons: let
$\psi_1$ and $\psi_2$ be two functions from $C_0^{\infty}(Q_T)$
such that $h(t)=\int_{\Omega}\psi_1(\bfx,t)\;
\rmd\bfx=\int_{\Omega}\psi_2(\bfx,t)\; \rmd\bfx$ for $t\in(0,T)$.
Denote by $G_1$, respectively $G_2$, the distribution, defined by
formula (\ref{4.13}) with $\psi=\psi_1$, respectively
$\psi=\psi_2$. Since $\supp(\psi_1-\psi_2)$ is a compact subset of
$Q_T$ and $\int_{\Omega}[\psi_1(\, .\, ,t)-\psi_2(\, .\, ,t)]\;
\rmd\bfx=0$ for all $t\in(0,T)$, there exists a function
$\bfphi\in\bfC_0^{\infty}(Q_T)$ such that $\div\bfphi=
\psi_1-\psi_2$ in $Q_T$. (See e.g.~\cite[Sec.~III.3]{Ga1} or
\cite{BoSo} for the construction of function $\bfphi$.) Then
\begin{displaymath}
\blangle G_1-G_2,h \brangle_{(0,T)}\ :=\ \bllangle
g,\psi_1-\psi_2\brrangle_{Q_T}\ =\ \bllangle
g,\div\bfphi\brrangle_{Q_T},
\end{displaymath}
which is equal to zero due to (\ref{4.12}). Formula (\ref{4.13})
and the identity $h(t)=\int_{\Omega}\psi(\bfx,t)\; \rmd\bfx$ show
that the distribution $g$ has the form (\ref{4.11}).

We have proven the theorem:

\begin{theorem} \label{T4.1}
The pressure, associated with a weak solution to the IBVP
(\ref{1.1})--(\ref{1.4}), is unique up to an additive distribution
of the form (\ref{4.11}).
\end{theorem}

\vspace{0pt} \noindent
{\bf 4.3. Projections $E^{1,2}_{\tau}$ and $E^{-1,2}_{\tau}$.} \
In this subsection, we introduce orthogonal projections
$E^{1,2}_{\tau}$ and $E^{-1,2}_{\tau}$ in $\Wt^{1,2}(\Omega)$ and
$\Wt^{-1,2}(\Omega)$, respectively, which further play an
important role in the proof of the existence of an associated
pressure.

$\Wt^{1,2}(\Omega)$ is a Hilbert space with the scalar product
$(\, .\, ,\, .\, )_{1,2}=\blangle(\cA_0+I)\, .\, ,\, .\
\brangle_{\tau}$, where $\cA_0$ is the operator $\cA$ from
paragraph 3.2, corresponding to $\nu=1$ and $\gamma=0$. Similarly,
$\Wt^{-1,2}(\Omega)$ is a Hilbert space with the scalar product
\begin{equation}
(\bfg,\bfh)_{-1,2}\ :=\ \blangle\bfg,(\cA_0+
I)^{-1}\bfh\brangle_{\tau}\ =\ \bigl((\cA_0+
I)^{-1}\bfg,(\cA_0+I)^{-1}\bfh\bigr)_{1,2}. \label{4.1}
\end{equation}
Denote by $E^{1,2}_{\tau}$ the orthogonal projection in
$\Wt^{1,2}(\Omega)$ that vanishes just on $\Wts^{1,2}(\Omega)$,
which means that
\begin{equation}
\ker E^{1,2}_{\tau}\ =\ \Wts^{1,2}(\Omega). \label{4.2}
\end{equation}
Denote by $E^{-1,2}_{\tau}$ the adjoint projection in
$\Wt^{-1,2}(\Omega)$. Applying (\ref{4.2}), one can verify that
the range of $E^{-1,2}_{\tau}$ is $\Wts^{1,2}(\Omega)^{\perp}$.

Let $\bfg\in\Wt^{-1,2}(\Omega)$ and $\bfpsi\in\Wt^{1,2}(\Omega)$.
Then, due to (\ref{4.1}) and the orthogonality of
$E^{1,2}_{\tau}$, we have
\begin{displaymath}
\blangle\bfg,E^{1,2}_{\tau}\bfpsi\brangle_{\tau}\ =\ \bigl(
(\cA_0+I)^{-1}\bfg,E^{1,2}_{\tau}\bfpsi\bigr)_{1,2}\ =\ \bigl(
E^{1,2}_{\tau}(\cA_0+I)^{-1}\bfg,\bfpsi\bigr)_{1,2}.
\end{displaymath}
However, the duality on the left hand side can also be expressed
in another way: \ using again (\ref{4.1}) and the fact that
$E^{-1,2}_{\tau}$ is adjoint to $E^{1,2}_{\tau}$, we get
\begin{displaymath}
\blangle\bfg,E^{1,2}_{\tau}\bfpsi\brangle_{\tau}\ =\ \blangle
E^{-1,2}_{\tau}\br\bfg,\bfpsi\brangle_{\tau}\ =\ \bigl(
(\cA_0+I)^{-1} E^{-1,2}_{\tau}\br\bfg,\bfpsi\bigr)_{1,2}.
\end{displaymath}
Thus, we obtain the important identity
\begin{equation}
E^{1,2}_{\tau}\br(\cA_0+I)^{-1}\ =\ (\cA_0+I)^{-1}
E^{-1,2}_{\tau}. \label{4.3}
\end{equation}
Applying (\ref{4.3}), we can now show that the projection
$E^{-1,2}_{\tau}$ is orthogonal in $\Wt^{-1,2}(\Omega)$. Indeed,
if $\bfg,\br\bfh\in\Wt^{-1,2}(\Omega)$ then
\begin{align*}
& \bigl(E^{-1,2}_{\tau}\br\bfg,\bfh\bigr)_{-1,2}\ =\
\bigl((\cA_0+I)^{-1}E^{-1,2}_{\tau}\br\bfg,
(\cA_0+I)^{-1}\bfh\bigr)_{1,2} \\
& \hspace{20pt} =\ \bigr(E^{1,2}_{\tau}\br (\cA_0+I)^{-1}\br
\bfg,(\cA_0+I)^{-1}\bfh\bigr)_{1,2}\ =\ \bigl((\cA_0+I)^{-1}
\bfg,E^{1,2}_{\tau}(\cA_0+I)^{-1}\bfh\bigr)_{1,2} \\
& \hspace{20pt} =\
\bigl((\cA_0+I)^{-1}\bfg,(\cA_0+I)^{-1}E^{-1,2}_{\tau}
\bfh\bigr)_{1,2}\ =\ \bigl(\bfg,E^{-1,2}_{\tau}\bfh\bigr)_{-1,2}.
\end{align*}
This verifies the orthogonality of projection $E^{-1,2}_{\tau}$.

Finally, we will show that if $\phi\in C^{\infty}_0(\Omega)$ then
\begin{equation}
E^{1,2}_{\tau}\nabla\phi\ =\ \nabla\phi \qquad \mbox{for all
$\phi\in C^{\infty}_0(\Omega)$}. \label{4.4}
\end{equation}
Thus, let $\phi\in C^{\infty}_0(\Omega)$. Then
$\nabla\phi\in\Wt^{1,2}(\Omega)$ and $\,
(\cA_0+I)\nabla\phi\equiv\nabla(-\Delta+I)\phi\in\Wts^{1,2}
(\Omega)^{\perp}$. Hence
\begin{displaymath}
E^{-1,2}_{\tau}(\cA_0+I)\nabla\phi\ =\ (\cA_0+I)\nabla\phi.
\end{displaymath}
Applying (\ref{4.3}), we also get
\begin{displaymath}
E^{-1,2}_{\tau}(\cA_0+I)\nabla\phi\ =\ (\cA_0+
I)E^{1,2}_{\tau}\nabla\phi.
\end{displaymath}
Since $\cA_0+I$ is a one-to-one operator from $\Wt^{1,2}(\Omega)$
to $\Wt^{-1,2}(\Omega)$, the last two identities show that
(\ref{4.4}) holds.

\vspace{4pt} \noindent
{\bf 4.4. Existence of an associated pressure.} \ In this
paragraph, we show that to every weak solution of the IBVP
(\ref{1.1})--(\ref{1.4}), an associated pressure exists and has a
certain structure.

Let $\bfu$ be a weak solution to the IBVP
(\ref{1.1})--(\ref{1.4}). Due to \cite[Lemma III.1.1]{Te},
equation (\ref{3.9}) is equivalent to
\begin{displaymath}
\bfu(t)-\bfu(0)+\int_0^t\cPs{2}\bigl[\cA\bfu+
\cB(\bfu,\bfu)-\bff\bigr]\; \rmd\tau\ =\ \bfzero
\end{displaymath}
for a.a.~$t\in(0,T)$. (As usually, we identify $\bfu(\, .\, ,t)$
and $\bfu(t)$.) Since $\bfu(t)$ and $\bfu(0)$ are in
$\Ls^2(\Omega)$, they coincide with $\cPs{2}\bfu(t)$ and
$\cPs{2}\bfu(0)$, respectively. (See paragraph 2.4.) Hence
\begin{displaymath}
\cPs{2}\biggl(\bfu(t)-\bfu(0)+\int_0^t\bigl[\cA\bfu+
\cB(\bfu,\bfu)-\bff\bigr]\; \rmd\tau\biggr)\ =\ \bfzero.
\end{displaymath}
Define $\bfF(t)\in\Wt^{-1,2}(\Omega)$ by the formula
\begin{equation}
\bfF(t)\ :=\ \bfu(t)-\bfu(0)+\int_0^t
\bigl[\cA\bfu+\cB(\bfu,\bfu)-\bff\bigr]\; \rmd\tau.
\label{4.6}
\end{equation}
Since $\langle\bfF(t),\bfpsi\rangle_{\tau}=\langle\cPs{2}
\bfF(t),\bfpsi\rangle_{\tau,\sigma}=0$ for all $\bfpsi\in
\Wts^{1,2}(\Omega)$, $\bfF(t)$ belongs to
$\Wts^{1,2}(\Omega)^{\perp}$. Hence
$E^{-1,2}_{\tau}\br\bfF(t)=\bfF(t)$ and
$(I-E^{-1,2}_{\tau})\br\bfF(t)=\bfzero$. Thus,
\begin{align*}
(I-E^{-1,2}_{\tau})\br\bfu(t) & -(I-E^{-1,2}_{\tau})\br\bfu(0) \\
& +\int_0^t(I-E^{-1,2}_{\tau})\br\bigl[\cA\bfu+
\cB(\bfu,\bfu)-\bff\bigr]\; \rmd\tau\ =\ \bfzero
\end{align*}
holds as an equation in $\Wt^{-1,2}(\Omega)$. Applying Lemma
III.1.1 from \cite{Te}, we deduce that
\begin{displaymath}
\bigl[(I-E^{-1,2}_{\tau})\br\bfu\bigr]'+(I-E^{-1,2}_{\tau})\br
\bigl[\cA\bfu+\cB(\bfu,\bfu)-\bff\bigr]\ =\ \bfzero.
\end{displaymath}
This yields
\begin{align}
\bfu'+\cA\bfu & +\cB(\bfu,\bfu)\ =\ \bff \nonumber \\
\noalign{\vskip 2pt}
& +E^{-1,2}_{\tau}[\bfu'+\cA\bfu+\cB(\bfu,\bfu)-\bff].
\label{4.7}
\end{align}
(Here, $[(I-E^{-1,2}_{\tau})\bfu]'$ and $\bfu'$ are the
distributional derivatives with respect to $t$ of
$(I-E^{-1,2}_{\tau})\bfu$ and $\bfu$, respectively, as functions
from $(0,T)$ to $\Wt^{-1,2}(\Omega)$.) Let
$\Omega_0\subset\subset\Omega$ be a non-empty domain. By Lemma
\ref{L2.2}, there exist unique $p_1(t)$, $p_{21}(t)$, $p_{22}(t)$,
$p_{23}(t)$ in $L^2_{loc}(\Omega)$ such that
\begin{equation}
\begin{array}{rl}
\blangle -E^{-1,2}_{\tau}\bfu(t),\bfpsi\brangle_{\tau}\ &=\
{\displaystyle -\int_{\Omega} p_1(t)\ \div\bfpsi\; \rmd\bfx,} \\
[10pt]
\blangle -E^{-1,2}_{\tau}\cA\bfu(t),\bfpsi \brangle_{\tau}\ &=\
{\displaystyle -\int_{\Omega} p_{21}(t)\ \div\bfpsi\; \rmd\bfx,}
\\ [10pt]
\blangle -E^{-1,2}_{\tau}\cB(\bfu(t),\bfu(t)),\bfpsi
\brangle_{\tau}\ &=\ {\displaystyle -\int_{\Omega} p_{22}(t)\
\div\bfpsi\; \rmd\bfx,}
\\ [10pt]
\blangle -E^{-1,2}_{\tau}\bff(t),\bfpsi \brangle_{\tau}\ &=\
{\displaystyle -\int_{\Omega} p_{23}(t)\ \div\bfpsi\; \rmd\bfx}
\end{array} \label{4.8}
\end{equation}
for a.a.~$t\in(0,T)$ and all $\bfpsi\in\Wtc^{1,2}(\Omega)$ and the
inequalities
\begin{equation}
\begin{array}{lll}
\|p_1(t)\|_{2;\, \Omega_R}\ & \leq\ c(R)\,
\|E^{-1,2}_{\tau}\bfu(t)\|_{-1,2}\ & \leq\ c(R)\, \|\bfu(t)\|_{-1,2}, \\
[5pt]
\|p_{21}(t)\|_{2;\, \Omega_R}\ & \leq\ c(R)\,
\|E^{-1,2}_{\tau}\cA\bfu(t)\|_{-1,2}\ & \leq\ c(R)\,
\|\cA\bfu(t)\|_{-1,2}, \\ [5pt]
\|p_{22}(t)\|_{2;\, \Omega_R}\ & \leq\ c(R)\,
\|E^{-1,2}_{\tau}\cB(\bfu(t),\bfu(t))\|_{-1,2}\ & \leq\ c(R)\,
\|\cB(\bfu(t),\bfu(t))\|_{-1,2}, \\ [5pt]
\|p_{23}(t)\|_{2;\, \Omega_R}\ & \leq\ c(R)\,
\|E^{-1,2}_{\tau}\bff(t)\|_{-1,2}\ & \leq\ c(R)\,
\|\bff(t)\|_{-1,2}
\end{array} \label{4.9}
\end{equation}
hold for all $R>0$ and a.a.~$t\in(0,T)$. Moreover,
$\int_{\Omega_0}p_1(t)\; \rmd\bfx=\int_{\Omega_0}p_{2i}(t)\;
\rmd\bfx=0$ ($i=1,2,3$) for a.a.~$t\in(0,T)$. Using the inequality
$\|\bfu(t)\|_{-1,2}\leq \|\bfu(t)\|_2$ and estimates (\ref{3.5}),
we get
\begin{equation}
\begin{array}{ll} p_1\hspace{4.1pt}\in L^{\infty}(0,T;\,
L^2(\Omega_R)), & p_{21}\in L^2(0,T;\, L^2(\Omega_R)), \\
[5pt] p_{22}\in L^{4/3} (0,T;\, L^2(\Omega_R)), \hbox to 10pt{} &
p_{23}\in L^2(0,T;\, L^2(\Omega_R)) \end{array} \label{4.10}
\end{equation}
for all $R>0$.

For a.a.~$t\in(0,T)$, the functions $p_1(t)$ and $p_{21}(t)$ are
harmonic in $\Omega$. This follows from the identities
\begin{align*}
\int_{\Omega} p_1(t)\, \Delta\phi\; \rmd\bfx\ &=\ -\blangle \nabla
p_1(t),\nabla\phi\brangle_{\tau}\ =\ \blangle
E^{-1,2}_{\tau}\bfu(t),\nabla\phi\brangle_{\tau}\ =\ \blangle
\bfu(t),E^{1,2}_{\tau}\br\nabla\phi\brangle_{\tau} \\
&=\ \blangle\bfu(t),\nabla\phi\brangle_{\tau}\ =\ \int_{\Omega}
\bfu(t)\cdot\nabla\phi\; \rmd\bfx\ =\ 0 \quad \mbox{(for all
$\phi\in C_0^{\infty}(\Omega)$).}
\end{align*}
(We have used (\ref{4.4}).) Hence, by Weyl's lemma, $p_1(t)$ is a
harmonic function in $\Omega$. The fact that $p_{21}(t)$ is
harmonic can be proved similarly.

Equation (\ref{4.7}) is an equation in $\Wt^{-1,2}(\Omega)$.
Applying successively each term in (\ref{4.7}) to the function of
the type $\bfvarphi(\bfx)\, \eta(t)$, where
$\bfvarphi\in\Wtc(\Omega)$ and $\eta\in C^{\infty}_0(0,T)$, using
formulas (\ref{4.8}), and denoting $p_2:=p_{21}+p_{22}+p_{23}$, we
obtain
\begin{gather*}
\int_0^T \int_{\Omega} \bigl[ -\bfu\cdot\bfvarphi\, \eta'(t)+
\nu\br\nabla\bfu:\nabla\bfvarphi\, \eta(t)+\bfu\cdot\nabla
\bfu\cdot \bfvarphi\, \eta(t) \bigr]\; \rmd\bfx\, \rmd
t+\int_0^T\int_{\Omega} \gamma\, \bfu\cdot\bfvarphi\, \eta(t)\;
\rmd S\, \rmd t \\
=\ \int_0^T \langle\bff,\bfvarphi\rangle_{\tau}\, \eta(t)\; \rmd
t-\int_0^T\int_{\Omega} p_1\ \div\bfvarphi\ \eta'(t)\; \rmd\bfx\,
\rmd t+\int_0^T\int_{\Omega} p_2\ \div \bfvarphi\ \eta(t)\;
\rmd\bfx\, \rmd t
\end{gather*}
for all functions $\bfvarphi\in\Wtc^{1,2}(\Omega)$ and $\eta\in
C^{\infty}_0((0,T))$. Since the set of all finite linear
combinations of functions of the type $\bfvarphi(\bfx)\, \eta(t)$,
where $\bfvarphi\in\Wtc^{1,2}(\Omega)$ and $\eta\in
C^{\infty}_0((0,T))$, is dense in $C^{\infty}_0\bigl((0,T);\,
\Wtc^{1,2}(\Omega)\bigr)$ in the norm of $W^{1,2}_0(0,T;\,
\Wt^{1,2}(\Omega))$, we also obtain the equation
\begin{gather}
\int_0^T \int_{\Omega} \bigl[ -\bfu\cdot\partial_t\bfphi+
\nu\br\nabla\bfu:\nabla\bfphi+\bfu\cdot\nabla\bfu\cdot\bfphi
\bigr]\; \rmd\bfx\, \rmd t+\int_0^T\int_{\Omega} \gamma\,
\bfu\cdot\bfphi\; \rmd S\, \rmd t \nonumber \\
=\ \int_0^T \langle\bff,\bfphi\rangle_{\tau}\; \rmd
t-\int_0^T\int_{\Omega} p_1\ \div\partial_t\bfphi\; \rmd\bfx\,
\rmd t+\int_0^T\int_{\Omega} p_2\ \div \bfphi\; \rmd\bfx\, \rmd t
\label{4.5}
\end{gather}
for all $\bfphi\in C^{\infty}_0\bigl((0,T);\,
\Wtc^{1,2}(\Omega)\bigr)$. Choosing particularly
$\bfphi\in\bfC^{\infty}_0(Q_T)$ and putting
\begin{equation}
p\ :=\ \partial_tp_1+p_2\ \equiv\ \partial_tp_1+p_{21}+p_{22}+
p_{23} \label{4.14}
\end{equation}
(where $\partial_tp_1$ is the derivative in the sense of
distributions), we observe that $(\bfu,p)$ is a distributional
solution of the system (\ref{1.1}), (\ref{1.2}) in $Q_T$.

The next theorem summarizes the results of this subsection:

\begin{theorem} \label{T4.2}
Let $T>0$ and $\Omega$ be a locally Lipschitz domain in $\R^3$,
satisfying condition (i) from subsection 1.1. Let $\bfu$ be a weak
solution to the Navier-Stokes IBVP (\ref{1.1})--(\ref{1.4}). Then
there exists an associated pressure $p$ in the form (\ref{4.14}),
where $p_1$, $p_{21}$, $p_{22}$, $p_{23}$ satisfy
(\ref{4.8})--(\ref{4.10}). Moreover,

\begin{list}{}
{\setlength{\topsep 2pt}
\setlength{\itemsep 0pt}
\setlength{\leftmargin 20pt}
\setlength{\rightmargin 0pt}
\setlength{\labelwidth 16pt}}

\item[1) ]
if $\, \Omega_0\subset\subset\Omega$ then the functions $p_1(t),\,
p_{21}(t),\, p_{22}(t),\, p_{32}(t)$ can be chosen so that they
satisfy the additional conditions
\begin{displaymath}
\int_{\Omega_0}p_1(t)\; \rmd\bfx\ =\ \int_{\Omega_0}p_{21}(t)\;
\rmd\bfx\ =\ \int_{\Omega_0}p_{23}(t)\; \rmd\bfx\ =\
\int_{\Omega_0}p_{23}(t)\; \rmd\bfx\ =\ 0,
\end{displaymath}

\item[2) ]
the functions $p_1(t)$ and $p_{21}(t)$ are harmonic in $\Omega$
for a.a.~$t\in(0,T)$,

\item[3) ]
the functions $\bfu$, $p_1$ and $p_2\equiv p_{21}+p_{22}+p_{23}$
satisfy the integral equation (\ref{4.5}) for all test functions
$\bfphi\in C^{\infty}_0\bigl((0,T;\, \Wtc^{1,2}(\Omega)\bigr)$.

\end{list}
\end{theorem}

\noindent
Note that if $\Omega$ is a bounded Lipschitz domain then the
choice $\Omega_0=\Omega$ is also permitted in statement 1) of
Theorem \ref{T4.2}.


\section{The case of a smooth bounded domain $\Omega$} \label{S5}

{\bf 5.1. Some results from paper \cite{AmEsGh}.} \ In this
section, we assume that $\Omega$ is a bounded domain in $\R^3$
with the boundary of the class $C^2$. We denote by $\Aq$ (for
$1<q<\infty$) the linear operator in $\Ls^q(\Omega)$ with the
domain defined by the equation
\begin{displaymath}
\Aq\bfv\ :=\ -\nu\, P_q\br\Delta\bfv
\end{displaymath}
for $\bfv\in D(\Aq)$, where
\begin{displaymath}
D(\Aq)\ :=\
\bigl\{\bfv\in\bfW^{2,q}(\Omega)\cap\Wts^{1,q}(\Omega);\
[\bbTd(\bfv)\cdot\bfn]_{\tau}+\gamma\br\bfv_{\tau}=\bfzero\
\mbox{on}\ \partial\Omega \bigr\}
\end{displaymath}
is the domain of operator $\Aq$. Recall that $\bbTd(\bfv)\equiv
2\nu\br\bbD(\bfv)$ is the dynamic stress tensor, induced by the
vector field $\bfv$, and $P_q$ is the Helmholtz projection in
$\bfL^q(\Omega)$. Operator $\Aq$ is usually called the {\it Stokes
operator} in $\Ls^q(\Omega)$. Particularly, if $q=2$ then $A_2$
coincides with the restriction of operator $\cA$, defined in
subsection 3.2, to $D(A_2)$. It is shown in the paper
\cite{AmEsGh} by Ch.~Amrouche, M.~Escobedo and A.~Ghosh that
$(-\Aq)$ generates a bounded analytic semigroup $\rme^{-\Aq t}$ in
$\Ls^q(\Omega)$. The next lemma also comes from \cite{AmEsGh}, see
\cite[Theorem 1.3]{AmEsGh}. It concerns the solution of the
inhomogeneous non--steady Stokes problem, given by the equations
\begin{equation}
\partial_t\bfu+\nabla \pi\ =\ \nu\Delta\bfu+\bfg \label{5.1}
\end{equation}
and (\ref{1.2}) (in $Q_T$), by the boundary conditions (\ref{1.3})
and by the initial condition (\ref{1.4}). The initial velocity
$\bfu_0$ is supposed to be from the space $\bfE_r^q(\Omega)$,
which is defined to be the real interpolation space $[D(\Aq),\,
\Ls^q(\Omega)]_{1/r,r}$. The problem (\ref{5.1}),
(\ref{1.2})--(\ref{1.3}) can also be equivalently written in the
form
\begin{equation}
\frac{\rmd\bfu}{\rmd t}+A_q\bfu\ =\ \bfg, \qquad \bfu(0)=\bfu_0,
\label{5.1a}
\end{equation}
which is the initial--value problem in $\Ls^q(\Omega)$. Although
the pressure $\pi$ does not explicitly appear in (\ref{5.1a}), it
can be always reconstructed in the way described in section
\ref{S4}.) The lemma says:

\begin{lemma} \label{L5.1}
Let $r,q\in(1,\infty)$, $T>0$, $\bfg\in L^r(0,T;\, \Ls^q(\Omega))$
and $\bfu_0\in\bfE_r^q(\Omega)$. Then the Stokes problem
(\ref{5.1}), (\ref{1.2}), (\ref{1.3}), (\ref{1.4}) has a unique
solution $(\bfu,\pi)$ in $\bigl[ W^{1,r}(0,T;\, \Ls^q(\Omega))\cap
L^r(0,T;\, \bfW^{2,q}(\Omega)) \bigr] \times L^r(0,T;\,
W^{1,q}(\Omega)/\R)$. The solution satisfies the estimate
\begin{equation}
\int_0^T  \|\partial_t\bfu\|_q^r\; \rmd
t+\int_0^T\|\bfu\|_{2,q}^r\; \rmd t+\int_0^T\|\pi\|_{1,q}^r\; \rmd
t\ \leq\ C\, \biggl( \int_0^T\|\bfg\|_q^r\; \rmd t+
\|\bfu_0\|_{\bfE_r^q(\Omega)}^r \biggr). \label{5.2}
\end{equation}
\end{lemma}

The proof is based on a more general theorem from the paper
\cite{GiSo} by Y.~Giga and H.~Sohr.

\vspace{4pt} \noindent
{\bf 5.2. Application of Lemma \ref{L5.1}.} \ If $\bfu$ is a weak
solution to the problem (\ref{1.1})--(\ref{1.4}) then, since
$\bfu\in L^{\infty}(0,T;\ \Ls^2(\Omega)) \cap L^2(0,T;\
\Wts^{1,2}(\Omega))$, one can verify that $\bfu\cdot\nabla\bfu\in
L^r(0,T;\, \bfL^q(\Omega))$ for all $1\leq r\leq 2$, $1\leq
q\leq\frac{3}{2}$, satisfying $2/r+3/q=4$. In order to be
consistent with the assumptions of Lemma \ref{L5.1} regarding $q$
and $r$, assume that $1<q<\frac{3}{2}$, $1<r<2$ and $2/r+3/q=4$.
Furthermore, assume that
$\bfu_0\in\bfE^q_r(\Omega)\cap\Ls^2(\Omega)$ and function $\bff$
on the right hand side of equation (\ref{1.1}) is in $L^r(0,T;\,
\bfL^q(\Omega)) \cap L^2(0,T;\, \Wt^{-1,2} (\Omega))$. Put
$\bfg:=P_q\bff-P_q(\bfu\cdot\nabla\bfu)$. Then, due to the
boundedness of projection $P_q$ in $\bfL^q(\Omega)$, $\bfg\in
L^r(0,T;\, \Ls^q(\Omega))$. Assume, moreover, that
$\bfu_0\in\bfE_r^q(\Omega)$. Now, we are in a position that we can
apply Lemma \ref{L5.1} and deduce that the linear Stokes problem
(\ref{5.1}), (\ref{1.2})--(\ref{1.4}) has a unique solution
$(\bfU,\pi)\in \bigl[ W^{1,r}(0,T;\, \Ls^q(\Omega))\cap L^r(0,T;\,
\bfW^{2,q}(\Omega)) \bigr] \times L^r(0,T;\, W^{1,q}(\Omega)/\R)$,
satisfying estimate (\ref{5.2}) with $\bfU$ instead of $\bfu$. In
order to show that the weak solution $\bfu$ of the nonlinear
Navier--Stokes problem (\ref{1.1})--(\ref{1.4}) satisfies the same
estimate, too, we need to identify $\bfu$ with $\bfU$.

\vspace{4pt} \noindent
{\bf 5.3. The identification of $\bfU$ and $\bfu$.} \ It is not
obvious at the first sight that $\bfU=\bfu$, because while $\bfU$
is a unique solution of the problem (\ref{5.1}),
(\ref{1.2})--(\ref{1.4}) in the class $W^{1,r}(0,T;\,
\Ls^q(\Omega))\cap L^r(0,T;\, \bfW^{2,q}(\Omega))$, $\bfu$ is only
known to be in $L^{\infty}(0,T;\ \Ls^2(\Omega)) \cap L^2(0,T;\
\Wts^{1,2}(\Omega))$. Nevertheless, applying the so called Yosida
approximation of the identity operator in $\Ls^q(\Omega)$, defined
by the formula $J^{(k)}_q:=(I+k^{-1}A_q)^{-1}$ (for $k\in\N$), in
the same spirit as in \cite{GiSo} or \cite{SoWa}, the equality
$\bfU=\bfu$ can be established. We explain the main steps of the
procedure in greater detail in the rest of this subsection.

At first, one can deduce from \cite[Section 3]{AmEsGh} that the
spectrum of $A_q$ is a subset of the interval $(0,\infty)$ on the
real axis, which implies that $J^{(k)}_q$ is a bounded operator on
$\Ls^q(\Omega)$ with values in $D(A_q)$. Obviously, $J^{(k)}_q$
commutes with $A_q$ and with $J^{(m)}_q$ (for $k,m\in\N$,
$k\not=m$) and $J^{(k)}_q=J^{(k)}_s$ on $\Ls^q(\Omega)\cap
\Ls^s(\Omega)$ (for $1<s<\infty$). If $q=2$ then $A_2$ is a
positive selfadjoint operator in $\Ls^2$, see \cite{BdV}.
Consequently, $J^{(k)}_2$ is a selfadjoint operator in
$\Ls^2(\Omega)$, too. Finally, it is proven in \cite[p.~246]{Yo}
that $J^{(k)}_q\bfv\to\bfv$ strongly in $\Ls^q(\Omega)$ for all
$\bfv\in\Ls^q(\Omega)$ and $k\to\infty$.

Consider (\ref{3.1}) with $\bfphi(\bfx,t)=[J^{(k)}_q\bfw](\bfx)\,
\vartheta(t)$, where $k\in\N$, $\bfw\in\Cns(\Omega)$ and
$\vartheta\in C^{\infty}_0\bigl([0,T)\bigr)$. In this case,
(\ref{3.1}) yields
\begin{align}
\int_0^T \int_{\Omega}\bigl[-\bfu & \cdot J^{(k)}_q\bfw\,
\vartheta'+ (\bfu\cdot\nabla\bfu)\cdot J^{(k)}_q\bfw\, \vartheta+
2\nu\br(\nabla\bfu)_s:(\nabla J^{(k)}_q\bfw)_s\bigr]\, \vartheta\;
\rmd\bfx\, \rmd t \nonumber \\ \noalign{\vskip-4pt}
& \hspace{14pt} +\int_0^T\int_{\partial\Omega}\gamma\br\bfu\cdot
J^{(k)}_q\bfw\, \vartheta\; \rmd S\, \rmd t \nonumber \\
\noalign{\vskip 2pt}
&=\ \int_0^T\int_{\Omega}\bff\cdot J^{(k)}_q\bfw\, \vartheta\;
\rmd\bfx\, \rmd t+ \int_{\Omega}\bfu_0\cdot J^{(k)}_q\bfw\,
\vartheta(0)\, \rmd\bfx. \label{5.3}
\end{align}
The integral of $(\bfu\cdot\nabla\bfu)\cdot J^{(k)}_q\bfw$ in
$\Omega$ can be rewritten as follows:
\begin{align*}
\int_{\Omega}(\bfu & \cdot\nabla\bfu)\cdot J^{(k)}_q\bfw\;
\rmd\bfx = \int_{\Omega}P_q(\bfu\cdot\nabla\bfu)\cdot
J^{(k)}_q\bfw\; \rmd\bfx = \int_{\Omega}
P_q(\bfu\cdot\nabla\bfu)\cdot J^{(k)}_2\bfw\; \rmd\bfx \\
&= \lim_{m\to\infty}\ \int_{\Omega}J^{(m)}_q
P_q(\bfu\cdot\nabla\bfu)\cdot J^{(k)}_2\bfw\; \rmd\bfx =
\lim_{m\to\infty}\ \int_{\Omega}J^{(k)}_2J^{(m)}_q
P_q(\bfu\cdot\nabla\bfu)\cdot \bfw\; \rmd\bfx \\
&= \lim_{m\to\infty}\ \int_{\Omega}J^{(k)}_qJ^{(m)}_q
P_q(\bfu\cdot\nabla\bfu)\cdot \bfw\; \rmd\bfx = \lim_{m\to\infty}\
\int_{\Omega}J^{(m)}_qJ^{(k)}_q P_q(\bfu\cdot\nabla\bfu)\cdot
\bfw\; \rmd\bfx \\
&= \int_{\Omega} J^{(k)}_q P_q(\bfu\cdot\nabla\bfu)\cdot \bfw\;
\rmd\bfx.
\end{align*}
This shows, except others, that the integrals of $\bfv_1\cdot
J^{(k)}_q\bfv_2$ and $J^{(k)}_q\bfv_1\cdot\bfv_2$ in $\Omega$ are
equal for $\bfv_1,\, \bfv_2\in\Ls^q(\Omega)$. The integrals of
$2\nu\br(\nabla\bfu)_s:(\nabla J^{(k)}_q\bfw)_s$ and
$\gamma\br\bfu\cdot J^{(k)}_q\bfw$ over $\Omega$ and
$\partial\Omega$, respectively, can be modified by means of the
identities:
\begin{align*}
\int_{\Omega} & 2\nu\br(\nabla\bfu)_s:(\nabla J^{(k)}_q\bfw)_s\;
\rmd\bfx+\int_{\partial\Omega}\gamma\bfu\cdot J^{(k)}_q\bfw\; \rmd
S \\
&=\ \int_{\Omega}2\nu\br\nabla\bfu : (\nabla J^{(k)}_q\bfw)_s\;
\rmd\bfx+\int_{\partial\Omega}\gamma\bfu\cdot J^{(k)}_q\bfw\; \rmd
S \\
&=\ \int_{\partial\Omega}2\nu\br\bfu\cdot[(\nabla
J^{(k)}_q\bfw)_s\cdot\bfn]\; \rmd S-\int_{\Omega}\nu\br\bfu\cdot
\Delta J^{(k)}_q\bfw\; \rmd\bfx+\int_{\partial\Omega}\gamma
\bfu\cdot J^{(k)}_q\bfw\; \rmd S \\
&=\ -\int_{\Omega}\nu\br\bfu\cdot \Delta J^{(k)}_q\bfw\; \rmd\bfx\
=\ \int_{\Omega}\bfu\cdot A_q J^{(k)}_q\bfw\; \rmd\bfx\ =\
\int_{\Omega}A_q\bfu\cdot J^{(k)}_q\bfw\; \rmd\bfx \\
&=\ -\int_{\Omega}J^{(k)} A_q\bfu\cdot\bfw\; \rmd\bfx\ =\
-\int_{\Omega}A_q J^{(k)}\bfu\cdot\bfw\; \rmd\bfx.
\end{align*}
Thus, we obtain from (\ref{5.3}):
\begin{align*}
\int_0^T \int_{\Omega}\bigl[-J^{(k)}_q\bfu & \cdot\bfw\,
\vartheta'+J^{(k)}_q P_q(\bfu\cdot\nabla\bfu)\cdot\bfw\,
\vartheta-\nu A_q J^{(k)}_q\bfu\cdot\bfw \bigr]\, \vartheta\;
\rmd\bfx\, \rmd t \nonumber \\ \noalign{\vskip 0pt}
&=\ \int_0^T\int_{\Omega} J^{(k)}_q\bff\cdot\bfw\, \vartheta\;
\rmd\bfx\, \rmd t+\int_{\Omega} J^{(k)}_q\bfu_0\cdot\bfw\,
\vartheta(0)\, \rmd\bfx.
\end{align*}
As $\bfw$ and $\vartheta$ are arbitrary functions from
$\Cns(\Omega)$ and $C^{\infty}_0\bigl([0,T)\bigr)$, respectively,
this shows that $J^{(k)}_q\bfu$ is a solution of the
initial--value problem
\begin{equation}
(J^{(k)}_q\bfu)'+A_q J^{(k)}_q\bfu\ =\ J^{(k)}_q\bfg, \qquad
J^{(k)}_q\bfu(\, .\, ,0)=J^{(k)}_q\bfu_0 \label{5.4}
\end{equation}
(which is a problem in $\Ls^q(\Omega)$) in the class
$W^{1,r}(0,T;\, \Ls^q(\Omega))\cap L^r(0,T;\,
\bfW^{2,q}(\Omega))$. Since $J^{(k)}_q\bfU$ solves the same
problem and belongs to the same class, we obtain the identity
$J^{(k)}_q\bfU(t)=J^{(k)}_q\bfu(t)$ for a.a.~$t\in(0,T)$.
Consequently, $\bfU(t)=\bfu(t)$ for a.a.~$t\in(0,T)$.

\vspace{4pt} \noindent
{\bf 5.4. The estimate of $\bfu$ and an associated pressure $p$.}
\ Since $\bfg=P_q\bff-P_q(\bfu\cdot\nabla\bfu)$, we can also write
equation (\ref{5.1}) in the form
\begin{align*}
\partial_t\bfu+\bfu\cdot\nabla\bfu\ &=\ -\nabla\pi+\nu\Delta\bfu+
\bff+(I-P_q)(-\bff+\bfu\cdot\nabla\bfu) \\
&=\ -\nabla(\pi+\zeta)+\nu\Delta\bfu+\bff,
\end{align*}
where $\nabla\zeta=(I-P_q)(\bfu\cdot\nabla\bfu-\bff)$. (The fact
that $(I-P_q)(\bfu\cdot\nabla\bfu-\bff)$ can be expressed in the
form $\nabla\zeta$ follows e.g.~from \cite[section III.1]{Ga1}.)
We observe that $p:=\pi+\zeta$ is a pressure, associated with the
weak solution $\bfu$. Since the pair $(\bfU,\pi)$ satisfies
(\ref{5.2}), $\bfu$ and $p$ satisfy the analogous estimate
\begin{align}
\int_0^T & \|\partial_t\bfu\|_q^r\; \rmd t+\int_0^T
\|\bfu\|_{2,q}^r\; \rmd t+\int_0^T\|p\|_{1,q}^r\; \rmd t
\nonumber \\
&\leq\ C\int_0^T\bigl( \|\bff\|_q^r+\| P_q(\bfu\cdot
\nabla\bfu)\|_q^r \bigr)\; \rmd t+C\,
\|\bfu_0\|_{\bfE_r^q(\Omega)}^r. \label{5.5}
\end{align}
We have proven the theorem:

\begin{theorem} \label{T5.1}
Let $\Omega$ be a bounded domain in $\R^3$ with the boundary of
the class $C^2$ and $T>0$. Let $1<q<\frac{3}{2}$, $1<r<2$,
$2/r+3/q=4$, $\bfu_0\in\bfE_r^q(\Omega)\cap \Ls^2(\Omega)$ and
$\bff\in L^r(0,T;\, \bfL^q(\Omega))\cap L^2(0,T;\,
\bfL^2(\Omega))$. Let $\bfu$ be a weak solution to the
Navier-Stokes IBVP (\ref{1.1})--(\ref{1.4}) and $p$ be an
associated pressure. Then $\bfu\in L^r(0,T;\ \bfW^{2,q}(\Omega))
\cap W^{1,r}(0,T;\, \bfL^q(\Omega))$ and $p$ can be identified
with a function from $L^r(0,T;\, L^{3q/(3-q)}(\Omega))$. The
functions $\bfu$, $p$ satisfy equations (\ref{1.1}), (\ref{1.2})
a.e.~in $Q_T$ and the boundary conditions (\ref{1.3}) a.e.~in
$\Gamma_T$. Moreover, they also satisfy estimate (\ref{5.5}).
\end{theorem}


\section{An interior regularity of the associated pressure} \label{S6}

{\bf 6.1. On previous results on the interior regularity of
velocity and pressure.} \ The next lemma recalls the well known
Serrin's result on the interior regularity of weak solutions to
the system (\ref{1.1}), (\ref{1.2}). (See e.g.~\cite{Oh},
\cite{Se} or \cite{Ga2}.) It concerns weak solutions in
$\Omega_1\times(t_1,t_2)$, where $\Omega_1$ is a sub-domain of
$\Omega$, independently of boundary conditions on $\\Gamma_T$.

\begin{lemma} \label{L6.1}
Let $\Omega_1$ be a sub-domain of $\Omega$, $0\leq t_1<t_2\leq T$
and let $\bfu$ be a weak solution to the system (\ref{1.1}),
(\ref{1.2}) with $\bff=\bfzero$ in $\Omega_1\times(t_1,t_2)$. Let
$\bfu\in L^r(t_1,t_2;\, \bfL^s(\Omega_1))$, where
$r\in[2,\infty)$, $s\in(3,\infty]$ and $2/r+3/s=1$. Then, if
$\Omega_2\subset\subset\Omega_1$ and $0<2\epsilon<t_2-t_1$,
solution $\bfu$ has all spatial derivatives (of all orders)
bounded in $\Omega_2\times(t_1+\epsilon,t_2-\epsilon)$.
\end{lemma}

Note that Lemma \ref{L6.1} uses no assumptions on boundary
conditions, satisfied by $\bfu$ on $\partial\Omega\times(0,T)$.
The assumption that $\bfu$ is a weak solution to the system
(\ref{1.1}), (\ref{1.2}) in $\Omega_1\times(t_1,t_2)$ means that
$\bfu\in L^{\infty}(t_1,t_2;\, \bfL^{\infty}(\Omega_1))\cap
L^2(t_1,t_2;\, \bfW^{1,2}(\Omega_1))$, $\div\bfu=0$ holds in the
sense of distributions in $\Omega_1\times(t_1,t_2)$ and $\bfu$
satisfies (\ref{3.1}) for all infinitely differentiable
divergence--free test functions $\bfphi$ that have a compact
support in $\Omega_1\times(t_1,t_2)$. (Then the last integral on
the left hand side and both integrals on the right hand side are
equal to zero.) Also note that applying the results of
\cite{Sereg}, one can add to the conclusions of Lemma \ref{L6.1}
that $\bfu$ is H\"older--continuous in
$\Omega_2\times(t_1+\epsilon,t_2-\epsilon)$. Lemma \ref{L6.1}
provides no information on the associated pressure $p$ or the time
derivative $\partial_t\bfu$ in $\Omega_2\times(t_1+\epsilon,t_2-
\epsilon)$. The known results on the regularity of $\bfu$ and
$\partial_t$ in $\Omega_2\times(t_1+\epsilon,t_2-\epsilon)$, under
the assumptions that $\bfu$ is a weak solution of (\ref{1.1}),
(\ref{1.2}) in $\Omega\times(t_1,t_2)$ satisfying the conditions
formulated in Lemma \ref{6.1} in $\Omega_1\times(t_1,t_2)$, say:

\begin{list}{}
{\setlength{\topsep 4pt}
\setlength{\itemsep 2pt}
\setlength{\leftmargin 17pt}
\setlength{\rightmargin 0pt}
\setlength{\labelwidth 8pt}}

\item[a)]
If $\Omega=\R^3$ then $p$, $\partial_t\bfu$ and all their spatial
derivatives (of all orders) are in $L^{\infty}(\Omega_2
\times(t_1+\epsilon,t_2-\epsilon)$, see \cite{Ne1}, \cite{Ne2}
\cite{SkaKu}.

\item[b)]
If $\Omega$ is a bounded or exterior domain $\R^3$ with the
boundary of the class $C^{2+(h)}$ for some $h>0$ and $\bfu$
satisfies the no--slip boundary condition $\bfu=\bfzero$ on
$\partial\Omega\times(0,T)$ then $p$ and $\partial_t\bfu$ have all
spatial derivatives (of all orders) in
$L^q(t_1+\epsilon,t_2-\epsilon;\, L^{\infty}(\Omega_2))$ for any
$q\in(1,2)$, see \cite{NePe1}, \cite{Ne1}, \cite{Ne2} or
\cite{SkaKu}.

\item[c)]
If $\Omega$ is a bounded domain $\R^3$ with the boundary of the
class $C^{2+(h)}$ for some $h>0$ and $\bfu$ satisfies the
Navier--type boundary conditions
\begin{displaymath}
\bfu\cdot\bfn=0, \qquad \curl\, \bfu\times\bfn=\bfzero \qquad
\mbox{on}\ \partial\Omega\times(t_1,t_2)
\end{displaymath}
then $p$ and $\partial_t\bfu$ have the same regularity in
$\Omega_2\times(t_1+\epsilon,t_2-\epsilon)$ as stated in item a),
see \cite{NeAlB}.

\end{list}

\noindent
In the proofs, it is always sufficient to show that the
aforementioned statements hold for $p$. The same statements on
$\partial_t\bfu$ follow from the fact that $\nabla p$ and
$\partial_t\bfu$ are interconnected through the Navier--Stokes
equation (\ref{1.1}).

\vspace{4pt} \noindent
{\bf 6.2. An interior regularity of $p$ in case of Navier's
boundary conditions.} \ We further assume that $\Omega$ and $T$
are as in Theorem \ref{T5.1} and $\bff=\bfzero$. The main result
of this section says:

\begin{theorem} \label{T6.1}
Let $\Omega$ and $T$ be as in Theorem \ref{T5.1} and
$\bff=\bfzero$. Let $\bfu$ be a weak solution to the problem
(\ref{1.1})--(\ref{1.4}). Let $\Omega_1$ be a sub-domain of
$\Omega$, $0<t_1<t_2\leq T$ and let $\bfu\in L^r(t_1,t_2;\,
\bfL^s(\Omega_1))$, where $r\in[2,\infty)$, $s\in(3,\infty]$ and
$2/r+3/s=1$. Finally, let $\Omega_3\subset\subset\Omega_1$ and
$0<\epsilon<t_2-t_1$. Then $p$ can be chosen so that all its
spatial derivatives (of all orders) are in
$L^4(t_1+\epsilon,t_2-\epsilon;\, L^{\infty}(\Omega_3))$.
Similarly, $\partial_t\bfu$ and all its spatial derivatives (of
all orders) are in $L^4(t_1+\epsilon,t_2-\epsilon;\,
\bfL^{\infty}(\Omega_3))$.
\end{theorem}

\begin{proof}
There exists $t_*\in(0,t_1)$ such that $\bfu(\, .\,
,t_*)\in\Wts^{1,2}(\Omega)\subset\bfE_r^q(\Omega)$ for all $r$ and
$q$, considered in Theorem \ref{T5.1}. Hence $\bfu\in L^r(t_*,T;\
\bfW^{2,q}(\Omega))\cap W^{1,r}(t_*,T;\, \bfL^q(\Omega))$ and $p$
can be chosen so that $p\in L^r(t_*,T;\, L^{3q/(3-q)}(\Omega))$.
Let $\epsilon$ and $\Omega_2$ be the number and domain,
respectively, given by Lemma \ref{L6.1}. We may assume that
$\Omega_2$ and $\Omega_3$ are chosen so that
$\emptyset\not=\Omega_3\subset\subset\Omega_2\subset\subset\Omega$.

Applying the operator of divergence to equation (\ref{1.1}), we
obtain the equation
\begin{equation}
\Delta p\ =\ -\nabla\bfu:(\nabla\bfu)^T, \label{6.1}
\end{equation}
which holds in the sense of of distributions in $Q_T$. Taking into
account that $p$ is at least locally integrable in
$\Omega_1\times(t_1,t_2)$, we obtain from (\ref{6.1}) that
\begin{displaymath}
\int_{t_1}^{t_2}\theta(t)\int_{\Omega_1} \bigl[ p\,
\Delta\varphi(\bfx)+\nabla\bfu:(\nabla\bfu)^T\,
\varphi(\bfx)\bigr]\; \rmd\bfx\, \rmd t\ =\ 0
\end{displaymath}
for all $\theta\in C^{\infty}_0((t_1,t_2))$ and $\varphi\in
C^{\infty}_0(\Omega_1)$. From this, we deduce that equation
(\ref{6.1}) holds in $\Omega_1$ in the sense of distributions at
a.a.~fixed time instants $t\in(t_1+\epsilon,t_2-\epsilon)$. Let
further $t$ be one of these time instants and let $t$ be also
chosen so that $\bfu(\, .\, ,t)\in\bfW^{2,q}(\Omega)$,
$\partial_t\bfu(\, .\, ,t)\in\bfL^q(\Omega)$ and $p(\, .\, ,t)\in
L^{3q/(3-q)}(\Omega)$. As $p(\, .\, ,t)\in L^1_{loc}(\Omega_1)$
and the right hand side of (\ref{6.1}) (at the fixed time $t$) is
infinitely differentiable in the spatial variable in $\Omega_2$,
the function $p(\, .\, ,t)$ is also infinitely differentiable in
$\Omega_2$, see e.g.~\cite{FraFio}.

Let $\bfx_0\in\Omega_3$ and $0<\rho_1<\rho_2$ be so small that
$B_{\rho_2}(\bfx_0)\subset\Omega_2$. Define an infinitely
differentiable non-increasing cut--off function $\eta$ in
$[0,\infty)$ by the formula
\begin{displaymath}
\eta(\sigma)\ \left\{ \begin{array}{ll} =1 & \mbox{for}\
0\leq\sigma\leq\rho_1, \\ [1pt] \in(0,1) & \mbox{for}\
\rho_1<\sigma<\rho_2, \\ [1pt] =0 & \mbox{for}\ \rho_2\leq\sigma.
\end{array} \right.
\end{displaymath}
Let $\bfx\in B_{\rho_1}(\bfx_0)$ and $\bfe$ be a constant unit
vector in $\R^3$. Then
\begin{align*}
\nabla_{\!\bfx}\br p(\bfx,t)\cdot\bfe\ &=\
\eta\bigl(|\bfx-\bfx_0|\bigr)\, \nabla_{\!\bfx}\br
p(\bfx,t)\cdot\bfe \\
&=\ -\frac{1}{4\pi}\int_{\R^3} \frac{1}{|\bfy-\bfx|}\
\Delta_{\bfy}\bigl[ \eta\bigl(|\bfy-\bfx_0|\bigr)\,
\nabla_{\!\bfy}\br p(\bfy,t)\cdot\bfe \bigr]\; \rmd\bfy.
\end{align*}
Particularly, this also holds for $\bfx=\bfx_0$:
\begin{align}
\nabla_{\!\bfx}\br p(\bfx,t)\cdot\bfe\,
\bigl|_{\bfx=\bfx_0}\bigr.\ &=\ -\frac{1}{4\pi}\int_{\R^3}
\frac{1}{|\bfy-\bfx_0|}\ \Delta_{\bfy}\bigl[
\eta\bigl(|\bfy-\bfx_0|\bigr)\, \nabla_{\!\bfy}\br
p(\bfy,t)\cdot\bfe \bigr]\; \rmd\bfy \nonumber \\
&=\ -\frac{1}{4\pi}\int_{\R^3} \frac{1}{|\bfy|}\
\Delta_{\bfy}\bigl[ \eta\bigl(|\bfy|\bigr)\, \nabla_{\!\bfy}\br
p(\bfx_0+\bfy,t)\cdot\bfe \bigr]\; \rmd\bfy \nonumber \\
&=\ -\frac{1}{4\pi}\, \bigl[
P^{(1)}(\bfx_0)+2P^{(2)}(\bfx_0)+P^{(3)}(\bfx_0) \bigr],
\label{6.3}
\end{align}
where
\begin{align*}
P^{(1)}(\bfx_0)\ &=\ \int_{B_{\rho_2}(\bfzero)} \frac{1}{|\bfy|}\
\Delta_{\bfy}\eta\bigl(|\bfy|\bigr)\, \bigl[ \nabla_{\!\bfy}\br
p(\bfx_0+\bfy,t)\cdot\bfe \bigr]\; \rmd\bfy, \\
P^{(2)}(\bfx_0)\ &=\ \int_{B_{\rho_2}(\bfzero)} \frac{1}{|\bfy|}\
\nabla_{\!\bfy}\br\eta\bigl(|\bfy|\bigr)\cdot\nabla_{\!\bfy}\br
\bigl[ \nabla_{\!\bfy}\br p(\bfx_0+\bfy,t)\cdot\bfe \bigr]\;
\rmd\bfy, \\
P^{(3)}(\bfx_0)\ &=\ \int_{B_{\rho_2}(\bfzero)}
\frac{\eta\bigl(|\bfy|\bigr)}{|\bfy|}\ \Delta_{\bfy}\br\bigl[
\nabla_{\!\bfy}\br p(\bfx_0+\bfy,t)\cdot\bfe \bigr]\; \rmd\bfy.
\end{align*}

\noindent
{\it The  estimate of $P^{(3)}(\bfx_0)$.} \ The estimate of the
last term is easy:
\begin{align}
\bigl| P^{(3)}(\bfx_0) \bigr|\ &=\ \biggl|
\int_{B_{\rho_2}(\bfzero)} \Bigl( \nabla_{\!\bfy}\,
\frac{\eta\bigl(|\bfy|\bigr)}{|\bfy|}\cdot\bfe
\Bigr)\, \Delta_{\bfy} p(\bfx_0+\bfy,t)\; \rmd\bfy \biggr|
\nonumber \\
&=\ \biggl| \int_{B_{\rho_2}(\bfzero)} \Bigl( \nabla_{\!\bfy}\,
\frac{\eta\bigl(|\bfy|\bigr)}{|\bfy|}\cdot\bfe \Bigr)\, \bigl[
\nabla_{\!\bfy}\br\bfu(\bfx_0+\bfy,t):\bigl(
\nabla_{\!\bfy}\br\bfu(\bfx_0+
\bfy,t)\bigr)^T\bigr]\; \rmd\bfy \biggr| \nonumber \\
&\leq\ c\, \int_{B_{\rho_2}(\bfzero)} \Bigl| \nabla_{\!\bfy}\,
\frac{\eta\bigl(|\bfy|\bigr)}{|\bfy|}\cdot\bfe \Bigr|\; \rmd\bfy\
\leq\ c. \label{6.4}
\end{align}

\noindent
{\it The estimate of $P^{(2)}(\bfx_0)$.} \ We can write
\begin{displaymath}
\frac{1}{|\bfy|}\ \nabla_{\!\bfy}\br\eta\bigl(|\bfy|\bigr)\ =\
\nabla_{\!\bfy}\br\cF\bigl(|\bfy|\bigr),
\end{displaymath}
where $\cF(s):=-\int_s^{\infty} \eta'(\sigma)/\sigma\; \rmd\sigma$
for $s\geq 0$. We observe that $\cF$ is constant on $[0,\rho_1]$,
equal to zero on $[\rho_2,\infty)$ and $\cF'(s)=\eta'(s)/s$ for
$s>0$. Thus, we have
\begin{align}
\bigl| P^{(2)}(\bfx_0) \bigr|\ &=\ \biggl|
\int_{B_{\rho_2}(\bfzero)} \nabla_{\!\bfy}\br
\cF\bigl(|\bfy|\bigr)\cdot \nabla_{\!\bfy}\br \bigl[
\nabla_{\!\bfy}\br p(\bfx_0+\bfy,t)\cdot\bfe \bigr]\;
\rmd\bfy \biggr| \nonumber \\
&=\ \biggl| \int_{B_{\rho_2}(\bfzero)}
\Delta_{\bfy}\cF\bigl(|\bfy|\bigr)\, \bfe\cdot \nabla_{\!\bfy}\br
p(\bfx_0+\bfy,t)\; \rmd\bfy \biggr|. \label{6.5}
\end{align}
The vector function $\Delta_{\bfy}\cF\bigl(|\bfy|\bigr)\, \bfe$
can be written in the form
\begin{equation}
\Delta_{\bfy}\cF\bigl(|\bfy|\bigr)\, \bfe\ =\
\nabla_{\!\bfy}\br\varphi(\bfy)+\bfw(\bfy), \label{6.6}
\end{equation}
where
\vspace{-8pt}
\begin{displaymath}
\varphi(\bfy)= \nabla_{\!\bfy}\br\cF\bigl(|\bfy|\bigr) \cdot\bfe,
\qquad \bfw(\bfy)=\Delta_{\bfy}\cF\bigl(|\bfy|\bigr)\, \bfe-
\nabla_{\!\bfy}\br\bigl[ \nabla_{\!\bfy}\br\cF\bigl(|\bfy|\bigr)
\cdot\bfe\bigr].
\end{displaymath}
The functions $\varphi$ and $\bfw$ are infinitely differentiable
in $\R^3$ and $\varphi=0$, $\bfw=\bfzero$ in $\R^3\smallsetminus
B_{\rho_2}(\bfzero)$. Since
\begin{displaymath}
\div\bfw\ =\ \nabla_{\!\bfy}\br\Delta_{\bfy}
\cF\bigl(|\bfy|\bigr)\cdot\bfe-\Delta_{\bfy}\br\bigl[
\nabla_{\!\bfy}\br\cF\bigl(|\bfy|\bigr) \cdot\bfe\bigr]\ =\ 0,
\end{displaymath}
(\ref{6.6}) in fact represents the Helmholtz decomposition of
$\Delta_{\bfy}\cF\bigl(|\bfy|\bigr)\, \bfe$ in
$B_{\rho_2}(\bfzero)$. Substituting from (\ref{6.6}) to
(\ref{6.5}), we obtain
\begin{align}
\bigl| P^{(2)}(\bfx_0) \bigr|\ &=\ \biggl| \int_{B_{\rho_2}
(\bfzero)} \bigl[\nabla_{\!\bfy}\br\varphi(\bfy)+\bfw(\bfy)\bigr]
\cdot\nabla_{\!\bfy}\br p(\bfx_0+\bfy,t)\; \rmd\bfy \biggr|
\nonumber \\
&=\ \biggl| \int_{B_{\rho_2} (\bfzero)}
\nabla_{\!\bfy}\br\varphi(\bfy) \cdot\nabla_{\!\bfy}\br
p(\bfx_0+\bfy,t)\; \rmd\bfy \biggr| \nonumber \\
&=\ \biggl| \int_{B_{\rho_2} (\bfzero)} \varphi(\bfy)\,
\Delta_{\bfy}p(\bfx_0+\bfy,t)\; \rmd\bfy \biggr| \nonumber \\
&=\ \biggl| \int_{B_{\rho_2} (\bfzero)} \varphi(\bfy)\, \bigl[
\nabla_{\!\bfy}\br\bfu(\bfx_0+\bfy,t):
\bigl(\nabla_{\!\bfy}\br\bfu(\bfx_0+\bfy,t)\bigr)^T \bigr]\;
\rmd\bfy \biggr| \nonumber \\
&\leq\ \int_{B_{\rho_2} (\bfzero)} |\varphi(\bfy)|\; \rmd\bfy \
\leq\ c. \label{6.11}
\end{align}

\noindent
{\it The  estimate of $P^{(1)}(\bfx_0)$.} \ Finally, we have
\begin{align}
P^{(1)}(\bfx_0)\ &=\ \int_{B_{\rho_2}(\bfzero)} \frac{1}{|\bfy|}\,
\nabla_{\!\bfy}\eta\bigl(|\bfy|\bigr)\cdot \nabla_{\!\bfy}\br
\bigl[\nabla_{\!\bfy}\br p(\bfx_0+\bfy,t)\cdot\bfe\bigr]\; \rmd\bfy
\nonumber \\
& \hspace{21pt} -\int_{B_{\rho_2}(\bfzero)}
\Bigl[\frac{\bfy}{|\bfy|^3}\cdot\nabla_{\!\bfy}\eta
\bigl(|\bfy|\bigr) \Bigr]\, \bigl[ \nabla_{\!\bfy}\br
p(\bfx_0+\bfy,t)\cdot\bfe\bigr]\; \rmd\bfy. \label{6.12}
\end{align}
The first integral coincides with the integral in the formula for
$P^{(2)}(\bfx_0)$ and it can be therefore treated in the same way.
The second integral on the right hand side of (\ref{6.12}) - let
us denote it by $P^{(1)}_2(\bfx_0)$ - represents the main
obstacle, which finally causes that $p$ and all its spatial
derivatives are only in $L^4(t_1+\epsilon,t_2-\epsilon;\,
L^{\infty}(\Omega_3))$ and not in
$L^{\infty}(t_1+\epsilon,t_2-\epsilon;\, L^{\infty}(\Omega_3))$,
as in the cases from items a) and c) in subsection 6.1. The
integral can be written in the form
\begin{align}
P^{(1)}_2(\bfx_0)\ &=\ \int_{B_{\rho_2}(\bfzero)}
\frac{\eta'\bigl(|\bfy|\bigr)} {|\bfy|^2}\,
\bfe\cdot\nabla_{\!\bfy}\br p(\bfx_0+\bfy,t)\; \rmd\bfy \nonumber
\\
&=\ \int_{\Omega} \frac{\eta'\bigl(|\bfy-\bfx_0|\bigr)}
{|\bfy-\bfx_0|^2}\, \bfe\cdot\nabla_{\!\bfy}\br p(\bfy,t)\;
\rmd\bfy. \label{6.13}
\end{align}
Now, we use the Helmholtz decomposition
\begin{equation}
\frac{\eta'\bigl(|\bfy-\bfx_0|\bigr)} {|\bfy-\bfx_0|^2}\, \bfe\ =\
\nabla_{\!\bfy}\br\psi(\bfy)+\bfz(\bfy), \label{6.14}
\end{equation}
in the whole domain $\Omega$, where
\begin{align*}
\Delta_{\bfy}\psi(\bfy) &=
\div\Bigl(\frac{\eta'\bigl(|\bfy-\bfx_0|\bigr)}
{|\bfy-\bfx_0|^2}\, \bfe\Bigr) = \Bigl(
\frac{\eta''\bigl(|\bfy-\bfx_0|\bigr)} {|\bfy-\bfx_0|^3}-
\frac{\eta'\bigl(|\bfy-\bfx_0|\bigr) }{|\bfy-\bfx_0|^4} \Bigr)\,
(\bfy-\bfx_0)\cdot\bfe && \mbox{for}\ \bfy\in\Omega, \\
\frac{\partial\psi}{\partial\bfn}(\bfy)\ &=\ 0 && \mbox{for}\
\bfy\in \partial\Omega.
\end{align*}
As $\bfz$ is divergence--free and its normal component on
$\partial\Omega$ is zero, and the integral of
$\nabla\psi\cdot\partial_t\bfu$ is zero, we get
\begin{align}
P^{(1)}_2(\bfx_0)\ &=\ \int_{\Omega} \bigl[
\nabla_{\!\bfy}\br\psi(\bfy)+\bfz(\bfy) \bigr]\cdot
\nabla_{\!\bfy}\br p(\bfy,t)\; \rmd\bfy\ =\ \int_{\Omega}
\nabla_{\!\bfy}\br\psi(\bfy)\cdot
\nabla_{\!\bfy}\br p(\bfy,t)\; \rmd\bfy \nonumber \\
&=\ \int_{\Omega}
\nabla_{\!\bfy}\br\psi(\bfy)\cdot\bigl[\partial_t\bfu+\bfu\cdot
\nabla\bfu-\nu\Delta\bfu\bigr](\bfy,t)\; \rmd\bfy
\nonumber \\
&=\ \int_{\Omega} \nabla_{\!\bfy}\br\psi(\bfy)\cdot\bigl[\bfu\cdot
\nabla\bfu-\nu\Delta\bfu\bigr](\bfy,t)\; \rmd\bfy
\label{6.15}
\end{align}
We have
\begin{align}
\biggl| \int_{\Omega} & \nabla_{\!\bfy}\psi\cdot(\bfu\cdot
\nabla\bfu)\; \rmd\bfy \biggr|\ =\ \biggl|\int_{\Omega}
\nabla_{\!\bfy}^2\psi : (\bfu\otimes\bfu)\; \rmd\bfy\biggr|\ \leq\
c\int_{\Omega}|\bfu|^2\; \rmd\bfy\ \leq\ c, \label{6.16} \\
\biggl| \int_{\Omega} & \nabla_{\!\bfy}\psi\cdot \nu\Delta\bfu\;
\rmd\bfy \biggr|\ =\ \biggl|\int_{\Omega}\nabla_{\!\bfy}\psi\cdot
\div\bbTd(\bfu)\; \rmd\bfy\biggr| \nonumber \\
&=\ \biggl|\int_{\partial\Omega}\nabla_{\!\bfy}\psi\cdot
[\bbTd(\bfu)\cdot\bfn]\; \rmd S-
\int_{\Omega}\nabla_{\!\bfy}^2\psi : \bbTd(\bfu)\; \rmd\bfy
\biggr| \nonumber \\
&=\ \biggl|-\int_{\partial\Omega}\nabla_{\!\bfy}\psi\cdot
\gamma\bfu\; \rmd S-\int_{\Omega}\nabla_{\!\bfy}^2\psi :
\nu\br(\nabla\bfu)_s\; \rmd\bfy \biggr| \nonumber \\
&\leq\ c\int_{\partial\Omega}|\bfu|\; \rmd S+\biggl|
\int_{\Omega}\nabla_{\!\bfy}^2\psi : \nu\br\nabla\bfu\; \rmd\bfy
\biggr|\ =\ c\int_{\partial\Omega}|\bfu|\; \rmd S+\biggl|
\int_{\Omega}(\partial_i\partial_j\psi)\, \nu\,
(\partial_ju_i)\; \rmd\bfy \biggr| \nonumber \\
&=\ c\int_{\partial\Omega}|\bfu|\; \rmd S+\biggl|
\int_{\partial\Omega}(\partial_j\psi)\, n_i\, \nu\,
(\partial_ju_i)\; \rmd\bfy \biggr| \nonumber \\
&=\ c\int_{\partial\Omega}|\bfu|\; \rmd S+\nu\,
\biggl|\int_{\partial\Omega}(\partial_j\psi)\,
[\partial_j(n_iu_i)-(\partial_jn_i)\, u_i]\; \rmd\bfy \biggr|
\nonumber \\
&=\ c\int_{\partial\Omega}|\bfu|\; \rmd S+\nu\,
\biggl|\int_{\partial\Omega}(\partial_j\psi)\, (\partial_jn_i)\,
u_i\; \rmd\bfy \biggr| \nonumber \\
&\leq\ c\int_{\partial\Omega}|\bfu|\; \rmd S\ \leq\ c\, \biggl(
\int_{\partial\Omega}|\bfu|^2\; \rmd S\biggr)^{\! 1/2}\ \leq\ c\,
\bigl( \|\bfu\|_2+\|\bfu\|_2^{1/2}\, \|\bfu\|_{1,2}^{1/2} \bigr)
\nonumber \\
&\leq\ c+c\, \|\bfu\|_{1,2}^{1/2}. \label{6.17}
\end{align}
The right hand side is in $L^4(t_1+\epsilon,t_2-\epsilon)$. We
have used the estimate
\begin{displaymath}
\bigl| \nabla\psi \bigr|_{1+(h)}\ \leq\ c\, \Bigl| \Bigl(
\frac{\eta''\bigl(|\bfy-\bfx_0|\bigr)} {|\bfy-\bfx_0|^3}-
\frac{\eta'\bigl(|\bfy-\bfx_0|\bigr) }{|\bfy-\bfx_0|^4} \Bigr)\,
(\bfy-\bfx_0)\cdot\bfe \Bigr|_{0+(h)}\ \leq\ c,
\end{displaymath}
where $|\, .\, |_{1+(h)}$ and $|\, .\, |_{0+(h)}$ are the norms in
the H\"older spaces $\bfC^{1+(h)}(\overline{\Omega})$ and
$C^{0+(h)}(\overline{\Omega})$, respectively, see \cite{Na}. The
integral of $|\bfu|^2$ on $\partial\Omega$ has been estimated by
means of \cite[Theorem II.4.1]{Ga1}.

We have shown that the norm of $\,
\nabla_{\!\bfx}p(\bfx,t)|_{\bfx=\bfx_0}\cdot\bfe\, $ in
$L^4(t_1+\epsilon,t_2-\epsilon)$ is finite and inde\-pen\-dent of
vector $\bfe$ and a concrete position of point $\bfx_0$ in domain
$\Omega_3$. Hence $\nabla p\in L^4(0,T;\,
\bfL^{\infty}(\Omega_3))$. From this, one can deduce that $p$ can
be chosen so that $p\in L^4(0,T;\, L^{\infty}(\Omega_3))$.
Similarly, dealing with $D^{\alpha}_{\bfx} p(\bfx,t)$, where
$\alpha\equiv(\alpha_1,\alpha_2,\alpha_3)$ is an arbitrary
multi-index, instead of $p(\bfx,t)$, we show that $D^{\alpha} p\in
L^4(0,T;\, L^{\infty}(\Omega_3)$, too. The proof is completed
\end{proof}

\vspace{4pt} \noindent
{\bf Acknowledgement.} \ The authors have been supported by the
Academy of Sciences of the Czech Republic (RVO 67985840) and by
the Grant Agency of the Czech Republic, grant No.~17-01747S.

\end{document}